\documentclass{amsart}

\usepackage[utf8]{inputenc}     
\usepackage{mathtools}
\usepackage{amsfonts}
\usepackage{amsthm}
\usepackage{amssymb}
\usepackage{amsmath}
\usepackage{tikz}
\usetikzlibrary{arrows,automata}
\usepackage{tikz-cd}
\usepackage[textcolor={black},textsize=tiny]{todonotes}
\usepackage{xcolor}

\usepackage{bm}
\usetikzlibrary{decorations.pathreplacing, positioning}

\usepackage{graphicx}
\usepackage{subcaption}

\newtheorem{theorem}{Theorem}

\newtheorem{lemma}[theorem]{Lemma}
\newtheorem{proposition}[theorem]{Proposition}

\newtheorem{corollary}[theorem]{Corollary}
\newtheorem*{notation*}{Notation}
\newtheorem*{remark*}{Remark}

\theoremstyle{definition}
\newtheorem{example}[theorem]{Example}
\newtheorem{definition}[theorem]{Definition}

\newcommand{\ord}[0]{\operatorname{ord}}        
\newcommand{\uval}[0]{ν_U}

\newcommand{\unor}[0]{\operatorname{nor}_U}

\newcommand{\nor}[0]{\operatorname{nor}}
\newcommand{\gtwo}{\widetilde{𝔾}}
\newcommand{\gfin}{\widetilde{𝔾}_{\text{fin}}}
\newcommand{\radix}{\cdot}

\newcommand{\gmean}[0]{\varphi}
\newcommand{\hmean}[0]{\overline{\varphi}}

\newcommand{\uspace}[0]{\mathbb{Z}_{U}}
\newcommand{\rval}[0]{\tilde{\Phi}}
\newcommand{\tval}[0]{\Phi}
\newcommand{\alphavec}[0]{{\bm\alpha}}
\newcommand{\alphanorm}[0]{{\gamma}}
\newcommand{\onevec}[0]{{\bm1}}

\newcommand{\ltlex}[0]{<_{\text{lex}}}

\newcommand{\period}[1]{\overbracket[0.7pt][1pt]{#1}}

\DeclareUnicodeCharacter{D7}{\times}            
\DeclareUnicodeCharacter{394}{\Delta}           
\DeclareUnicodeCharacter{3B1}{\alpha}           
\DeclareUnicodeCharacter{3B2}{\beta}            
\DeclareUnicodeCharacter{3B4}{\delta}           
\DeclareUnicodeCharacter{3B5}{\varepsilon}      
\DeclareUnicodeCharacter{3B7}{\eta}             
\DeclareUnicodeCharacter{3BB}{\lambda}          
\DeclareUnicodeCharacter{3BD}{\nu}              
\DeclareUnicodeCharacter{3C0}{\pi}              
\DeclareUnicodeCharacter{3C1}{\rho}             
\DeclareUnicodeCharacter{3C3}{\sigma}           
\DeclareUnicodeCharacter{2026}{\ldots}          
\DeclareUnicodeCharacter{2102}{\mathbb{C}}      
\DeclareUnicodeCharacter{210D}{\mathbb{H}}      
\DeclareUnicodeCharacter{2113}{\ell}            
\DeclareUnicodeCharacter{2115}{\mathbb{N}}      
\DeclareUnicodeCharacter{211A}{\mathbb{Q}}      
\DeclareUnicodeCharacter{211D}{\mathbb{R}}      
\DeclareUnicodeCharacter{2124}{\mathbb{Z}}      
\DeclareUnicodeCharacter{2192}{\rightarrow}     
\DeclareUnicodeCharacter{21A6}{\mapsto}         
\DeclareUnicodeCharacter{21BE}{\upharpoonright} 
\DeclareUnicodeCharacter{2200}{\forall}         
\DeclareUnicodeCharacter{2203}{\exists}         
\DeclareUnicodeCharacter{2205}{\varnothing}     
\DeclareUnicodeCharacter{2208}{\in}             
\DeclareUnicodeCharacter{2209}{\notin}          
\DeclareUnicodeCharacter{2211}{\sum}            
\DeclareUnicodeCharacter{2216}{\setminus}       
\DeclareUnicodeCharacter{2218}{\circ}           
\DeclareUnicodeCharacter{221E}{\infty}          
\DeclareUnicodeCharacter{2227}{\wedge}          
\DeclareUnicodeCharacter{2229}{\cap}            
\DeclareUnicodeCharacter{222A}{\cup}            
\DeclareUnicodeCharacter{223C}{\sim}            
\DeclareUnicodeCharacter{2245}{\cong} 		
\DeclareUnicodeCharacter{2254}{\vcentcolon=}    
\DeclareUnicodeCharacter{2260}{\neq}            
\DeclareUnicodeCharacter{2261}{\equiv}          
\DeclareUnicodeCharacter{2282}{\subset}         
\DeclareUnicodeCharacter{2286}{\subseteq}       
\DeclareUnicodeCharacter{22C3}{\bigcup}         
\DeclareUnicodeCharacter{22C5}{\cdot}           
\DeclareUnicodeCharacter{22EF}{\cdots}          
\DeclareUnicodeCharacter{2308}{\lceil}          
\DeclareUnicodeCharacter{2309}{\rceil}          
\DeclareUnicodeCharacter{230A}{\lfloor}         
\DeclareUnicodeCharacter{230B}{\rfloor}         
\DeclareUnicodeCharacter{27E8}{\langle}         
\DeclareUnicodeCharacter{27E9}{\rangle}         
\DeclareUnicodeCharacter{2A7D}{\leqslant}       
\DeclareUnicodeCharacter{2A7E}{\geqslant}       
\DeclareUnicodeCharacter{1D53E}{\mathbb{G}}     
\DeclareUnicodeCharacter{1D543}{\mathbb{L}} 	
\DeclareUnicodeCharacter{1D54B}{\mathbb{T}}     
\newcommand{\fin}[0]{\operatorname{Fin}}

\begin{document}
\title[From some  Pisot numerations to topological groups]{From some  Pisot numerations\\ to topological groups}
\thanks{The first author was supported by the Agence Nationale de la
  Recherche through the project ``SymDynAr'' (ANR-23-CE40-0024-01).  }
\author[O.~Carton]{Olivier Carton}
\address{Institut universitaire de France et IRIF, Université Paris-Cité, France}
\author[J.~Sudbery]{Jake Sudbery}
\address{Department of Statistical Science, University College London, UK }
\author[R. Yassawi]{Reem Yassawi}
\address{School of Mathematical Sciences, Queen Mary University of London, UK}
\email{Olivier.Carton@irif.fr}
\email{jake.sudbery.23@ucl.ac.uk}
\email{r.yassawi@qmul.ac.uk}
\thanks{}
\date{\today}
\maketitle

\date{\today}
\maketitle

\begin{abstract}
  A Pisot numeration system~$U$ for~$ℕ$ is a sequence of natural numbers
  generated by an integral homogeneous linear recurrence whose
  characteristic polynomial is the minimal polynomial of a Pisot number.
  The purpose of this paper is to introduce the analogue of the group of
  $p$-adic integers for such numerations when they \emph{preserve zeros},
  which is equivalent to the ``Condition F'' introduced by Frougny and
  Solomyak for $β$-numerations. We show that these topological groups~$ℤ_U$
  project homomorphically onto a torus. Equipping $ℤ_U$ with the
  appropriate topology, we also show that if $U$ is unimodular, then $ℤ_U$
  is continuously isomorphic to a torus.
\end{abstract}

\section{Introduction}

The additive group $ℤ_p$ of $p$-adic integers can be constructed as the
inverse limit of the finite groups~$ℤ/p^nℤ$, but it can also be constructed
as a topological completion of~$ℕ$.  This completion is achieved through an
ultrametric which is based on the base-$p$ expansions of natural integers.
The purpose of this article is to construct the analogue of this group for
Pisot numerations.  Recall that a Pisot numeration replaces the powers
of~$p$, used in base-$p$ expansions, by a sequence of integers satisfying a
linear recurrence defined by the minimal polynomial of a Pisot number,
i.e., an algebraic integer all of whose conjugates lie in the interior of
the unit disk.  The iconic, and simplest, numeration of this type is the
Zeckendorf numeration, which uses the sequence of Fibonacci numbers,
associated to the golden mean.  In such a numeration each natural number
has a unique \emph{normalised} expansion, i.e., the one which is obtained
using the greedy algorithm; see Section~\ref{sec:Pisot-numeration} for
definitions and details.

In this article we introduce, for each Pisot numeration~$U$ satisfying a
mild condition, a group~$\uspace$ which is the analogue of~$ℤ_p$.  It can
be viewed as a topological completion of the set of normalised
$U$-expansions of the natural numbers, as each element of~$\uspace$ can be
represented by an infinite normalised sequence of digits.  Our construction
of~$\uspace$ does not follow the same strategy as for~$ℤ_p$ for the
following reason.  When two base-$p$ expansions are added, the carries
propagate to the left.  This is reflected in the fact that the $p$-adic
valuation~$ν_p$ satisfies the classical inequality
$ν_p(x+y) ⩾ \min(ν_p(x), ν_p(y))$.  This property no longer holds for Pisot
numerations: carries may propagate both to the left and to the right, as is
already the case for the Zeckendorf numeration.  For this reason, our
construction is based on a \emph{pseudo-valuation}~$ν_U$ which satisfies
the weaker inequality $ν_U(x+y) ⩾ \min(ν_U(x), ν_U(y))-K$ for some
constant~$K$; we say that numerations that satisfy this \emph{preserve
  zeros}.  We also equip our group~$\uspace$ with a topology, but as a
result of the nonstandard propagation of the carries, $\uspace$ is not a
profinite group with a zero-dimensional topology.

Our main result is that for a given Pisot numeration, the group that
we construct is also topologically an $n$-torus~$𝕋$, in the presence of
mild hypotheses.
\begin{theorem}
  Let $U = (u_n)_{n⩾0}$ be a numeration associated with a Pisot number~$β$
  of degree~$d$, with standard initial conditions. If $U$ preserves zeros,
  then there is a continuous group epimorphism $\tval: \uspace → 𝕋^{d-1}$.
  If $β$ is unimodular, then $\tval$ is an isomorphism.
\end{theorem}
The torus $𝕋^{d-1}$ is realised as $(ℂ^k× ℝ^ℓ)/𝕃$, where the Pisot number
has $2k$ complex and $ℓ$ real conjugates, and where $𝕃$ is a regular
lattice, defined in Section~\ref {sec:IsomorphisWithTorus}.  When defining
a numeration~$U$, one degree of freedom is the choice of  initial
conditions. We found that the right choice of initial conditions, which we
call \emph{standard}, is relevant. By taking other initial conditions
we can create coverings of the space instead of tilings, or indeed, small
initial conditions so that the image of~$\rval$ is not surjective; see
Example~\ref{ex:initial-conditions}. 
  
Much of the novelty in our result is the development of the topological
group~$\uspace$.  We start with the group $𝔾$ of bounded integer-valued
sequences, equipped with componentwise addition, and define a kind of
valuation. Through this we define a subgroup $ℍ_U$ of expansions of~$0$,
and we set $\uspace = 𝔾/ℍ_U$; details are in
Section~\ref{sec:U-adic-integers}. The definition of~$\tval$ is then
 is adapted from what already exists in the tiling and substitutional literature,
as we summarise below, although our map is in addition now a group
homomorphism.

A Pisot numeration comes with its associated Pisot number~$β$, and the
study of $β$-numerations is well-established. We show in
Section~\ref{sec:condition-F} that a Pisot numeration preserves zeros if
and only if the associated $β$-numeration satisfies \emph{Condition~F},
introduced by Frougny and Solomyak in \cite{Frougny-Solomyak-1992}.
Frougny and Solomyak identified families of Pisot numbers that satisfy
Condition~F, and Akiyama \cite{Akiyama-1998} characterised the
unimodular~$β$ which satisfy it, and also gave an algorithm which decides
whether or not a given Pisot number satisfies Condition~F. This then allows us to
decide when a numeration preserves zeros.

The closest result in the literature which compares to ours, but in the
$\beta$-numeration setting, is that of Vershik~\cite{Vershik-1991}. There
he considers $\beta$-expansions over powers of the golden ratio. He
constructs a completion of the set of those numbers, and shows, using
character theory, that this completion is isomorphic to the field of real
numbers, but ``with an unusual embedding of the set of powers of the golden
mean''.  Vershik also writes that one can approach this via the ``Fibonacci
decomposition of naturals into a sum of Fibonacci numbers (Zeckendorf
theorem)...but it does not explain the algebraic structure of this
construction.''  Here we take this approach and we make explicit the
algebraic structure of this construction.

There is also related work on characterising when an odometer defined by a
numeration admits addition by 1 as a continuous map. In
\cite{Grabner-Liardet-Tichy}, Grabner, Liardet and Tichy define a
\emph{G-adic} space, where $G = (G_n)_{n ⩾ 0}$ is a strictly increasing
sequence of positive integers, and study the addition of 1 in this space,
characterising $G$ for which this map is continuous. No group structure is
defined for these odometers.  The authors also study the case when $G$ is
obtained by a linear recurrence, and show that the adic system is
measure-theoretically isomorphic to a group rotation given a certain
\emph{Hypothesis B}. They ask whether this hypothesis is equivalent to
Condition~F, which would make this result equivalent to that of Solomyak
\cite{Solomyak-1992}.

Several authors focus on extending the Zeckendorf numeration to the
integers, or to a group. In~\cite{Rittaud-Vivier-2012}, Rittaud and Vivier
define a topological group  specifically for this numeration, by defining a valuation, and then a group which turns out to be  identical to ours in that particular case.   In other recent work,
connections are made between Zeckendorf numerations for $\mathbb{Z}^2$ and
Wang tiles. In~\cite{Labbe-Lepsova-2021}, Labbé and Lepšová show the
existence of a Wang tiling whose $(m,n)$-th entry is the output of an
automaton when the input is a representation of $(m,n)$ in this
numeration. In~\cite{Labbe-Lepsova-2023}, they show that this numeration is
the Zeckendorf analogue of the two's complement numeration system, which
gives a base-2 expansion for negative integers. They also generalise this
to a two-sided approach of Dumont-Thomas numerations
in~\cite{Labbe-Lepsova-2024}.
  
Our topological groups have familiar geometric representations, as
well-known tiles that come from substitutions (see below) and which are
fundamental domains of Euclidean space; see our examples. In
\cite{Akiyama-1997}, Akiyama also studied tilings generated by
$β$-numerations satisfying Condition~F. He writes that his work treats
these tilings in the context of numerations and not substitutions, as this
method gives a clearer understanding of the universal phenomena of these
tilings.  That these tilings already come from the underlying numeration of
the substitution is pertinent.  One difference between $U$ and a
$β$-numeration is that a $β$-expansion has both an integral and a
fractional part, the latter to the right of the radix point, the former to
the left. To generate his tilings, Akiyama considers valid finite
expansions, and maps them in a now standard way, as done by Rauzy
in~\cite{Rauzy-1982}, to a Euclidean space. These techniques, or some mild
modification, are a widely-used method to obtain these fractals in Euclidean
space, see, e.g., \cite{Canterini-Siegel-TAMS}. Indeed this is the technique
we also use to map $ℤ_U$ to the torus.  The integral part of the
expansion corresponds to a translation, while the fractional part generates
one tile. Because Akiyama considers these two parts in tandem, and because
different integral parts have different fractional continuations, he
obtains a tiling of the space using finitely many tiles, whereas we only
use one.

We also emphasise that our results concern numerations and not substitutions.
Nevertheless we can transfer our results to a family of substitutions,
whose fixed point has a shift-orbit that can be coded in an ordered way
using the orbit of addition by 1 in $ℤ_U$. These are the
\emph{$β$-substitutions}.  In \cite{Solomyak-1992}, Solomyak showed
that $β$-substitution shifts where $β$ satisfies the Condition~F
have discrete spectrum; this was extended more recently by Barge to all
$β$-substitutions \cite{Barge-2018}. We can use our setting to deduce
that these shifts are almost automorphic; this will be done in future
work. There are several works linking numerations to substitutions, mainly
via their Bratteli-Vershik representations as generalised odometers or the
Dumont-Thomas numeration system.  Some references are \cite{Rauzy-1982,
  Canterini-Siegel-TAMS, BBLT-2006, Thuswaldner-2006, Surer-2020}. This
list is far from exhaustive, and many works use some variation of Rauzy's
coding.

We mention other work where this kind of coding of a $\beta$-shift to a
torus is used. In~\cite{Vershik-1991}, Vershik also studied the Fibonacci
hyperbolic toral automorphism, and uses the same coding techniques as above
to show that this automorphism is arithmetically isomorphic to a
$\beta$-shift; this approach is generalised in \cite{Sidorov-Vershik-1998}.
Sidorov and Vershik work with $\beta$-numerations and Pisot numbers of
degree~2, characterising when one can establish arithmetic codings from a
two-sided $β$-shift of expansions of positive real numbers to a
hyperbolic toral automorphism.  Ex post facto, they can
impose a group structure by identifying points in shift space via the
arithmetic coding. Sidorov~\cite{Sidorov-2003} generalises this to higher
degree polynomials. There is also related work by Schmidt in
\cite{Schmidt-2000}.  Our work is transversal to this, in the sense that
instead of multiplication, the underlying dynamics is addition.

While the groups~$\uspace$ that we introduce are the analogue for Pisot
numerations of the additive group structure of~$ℤ_p$, the multiplicative
structure of~$ℤ_p$ is ignored. So far it seems difficult to turn our groups
into rings.  The usual multiplication of integers is not compatible with
the topology we use in the sense that it does not preserve zeros.  Knuth
introduced in \cite{Knuth-1988} another multiplication on integers based on
the Zeckendorf numeration, but it is not distributive with respect to
addition.

In Section~\ref{sec:preliminaries}, we set up notation. In
Section~\ref{sec:U-adic-integers}, we define zero preservation, and the
$U$-adic integers, and show that their elements are infinite normalised
sequences for~$U$.  Once we have established $\uspace$, we introduce a
topology and link it to a torus in Section~\ref{sec:IsomorphisWithTorus},
using classical techniques that are used in many of the above cited
articles. We also show that in the case where $U$ is unimodular, $\uspace$
is continuously group isomorphic to a torus. In
Section~\ref{sec:condition-F}, we show that the condition that $U=U_β$
preserves zeros is equivalent to the condition that the $β$-numeration
satisfy Condition~F.  Throughout we are careful to isolate the conditions
that we need.

\section{Preliminaries}\label{sec:preliminaries}
\subsection{Notation}

In this paper, we consider sequences $g=(g_n)_{n⩾ 0}$ over an alphabet
contained in $ℤ$ and containing 0.  If $0 ⩽ m ⩽ n$ are two non-negative
integers and $g$ is a sequence then define $g_{[m:n]} ≔ (g_k)_{k=m}^n$ and
$g_{[m:n)}≔ (g_k)_{k=m}^{n-1}$.  These are called \emph{factors} of $g$.
For any $n ⩾ 0$, $g_{[0:n]}$ is a \emph{head} of $g$ and $g_{[n:∞)}$ is a
\emph{tail} of~$g$.  A sequence is \emph{finite}, if it has a tail of
zeros, i.e. $g_n=0$ for all $n$ large.  We will consider both left-infinite
sequences $⋯ g_2g_1 g_0$ and right-infinite sequences $g_1g_2g_3⋯$.  We
will sometimes use the radix point ``$⋅ $'' to clarify the context, e.g.,
$g⋅$ represents a left-infinite sequence.  When we want to emphasise that
there is no tail of zeros, we call the sequence a \emph{word}. Let
$\period{w}$ denote a one-sided infinite sequence obtained by repeating the
word~$w$, so that, eg, a finite sequence is $\period{0}g_n⋯ g_0$.  As our
words will represent expansions of numbers, then, as in integral base
arithmetic, we will often identify a finite word~$w$ with any word
which has $w$ as a prefix, followed by any number of zeros, including an
infinite tail of zeros. An exception is when we concatenate words, as then
inserting zeros affects the resulting word.

Given $d ∈ ℤ$, let $\bar{d}$ denote the integer $-d$.  We also use the bar
to denote complex conjugates of numbers, but this latter notation will not
be confusing as the different usages are disjoint.  The set of finite
sequences on the alphabet~$D$ is denoted $D^*$. The \emph{length } of
$w=(d_i)_{i=1}^n$ is $n$ and is denoted $|w|$.

\subsection{Pisot numerations}\label{sec:Pisot-numeration}

We will work with numeration systems $U=(u_n)_{n⩾ 0}$ with $u_0 = 1$, where
$(u_n)_{n⩾ 0}$ is a sequence of natural numbers, and where each natural
number can be represented using finite sequences of \emph{digits} from a
finite set~$D ⊆ ℤ$, i.e., for each $n$ there is a natural number~$k$ and
$d_k,…, d_0$ from $D$ such that $n = ∑_{i=0}^k d_iu_i$. We use the
convention that if some $u_n=0$, then $d_n=0$.  In this paper, we only
consider numerations $U=(u_n)_{n⩾ 0}$ which satisfy a linear recurrence.

Furthermore, we only consider Pisot numerations, defined as follows.
Recall that $β>1$ is a Pisot number if it is algebraic over~$ℚ$ and if all
its Galois conjugates have modulus strictly less than one; for an
exposition see \cite{Parry-1960}.  Let
$P(X)= X^d-a_{d-1}X^{d-1}- ⋯ - a_1X-a_0$ be irreducible over~$ℚ$ and have a
leading root~$β$ which is Pisot.  We say that the sequence of natural
numbers $U = (u_n)_{n⩾0}$ is a \emph{Pisot} numeration system (of degree
$d$) if it is defined by initial conditions $u_0, u_{-1},…, u_{-d+1}$
and if for some $n ⩾ 1$ and integers $a_0, … ,a_{d-1}$,
\begin{displaymath}
  u_n = a_{d-1}u_{n-1}+ ⋯ +a_1u_{n-d+1}+a_0u_{n-d}.
\end{displaymath}
For convenience we give the initial conditions with non-positive indices.
In this paper, unless otherwise indicated, we will generate $(u_n)_{n⩾ 0}$
using the initial conditions
\begin{displaymath}
  u_{-1}=u_{-2}= ⋯ =u_{-d+1}=0 \text{ and } u_0=1.
\end{displaymath}
With these initial conditions we call the numeration \emph{standard}. Some
of our results are specific to these initial conditions; in places where
this is the case, we will indicate this explicitly. Since $P$ is
irreducible, it is also sometimes called the minimal polynomial for the
recurrence.  If we need to specify the leading root~$β$, we will say that
$U$ is associated to~$β$.  We will not require that the coefficients~$a_i$
are nonnegative, but we will require that the sequence $(u_n)_{n⩾ 0}$
generated by the recurrence be strictly positive.  For such Pisot
numerations, the sequence $(u_n)_{n ⩾ 0}$ is strictly increasing for large
$n$; see Lemma~\ref{lem:boundun}.

For these numeration systems, every natural number~$n$ has a
\emph{normalised} expansion, which is the greatest expansion for the
lexicographic order, and which is obtained using the greedy algorithm. Note
that for $U$-numerations, the least significant digit is the rightmost
digit, i.e., we write these normalised expansions extending to the left,
$(n)_U ≔ d_k ⋯ d_0$, where $d_k ≠ 0$ is the most significant digit
in the normalised expansion $(n)_U$ of $n ∈ ℕ$. Despite the
definition of~$(n)_U$, we are permitted to sometimes pad $(n)_U$ with
leading zeros, eg, when we have to compare the expansions of several
integers at a time.  Note that a finite sequence $d_k ⋯ d_0$ over~$ℕ$ is a
normalised expansion of a natural number if and only if
$∑_{i=0}^ℓd_iu_i< u_{ℓ+1}$ for each $ℓ ⩽ k$. We call such a finite sequence
\emph{$U$-normalised}.

Under the hypothesis that the ratio $u_{n+1}/u_n$ is bounded by a constant
for all $n$, which is the case for Pisot numerations, see, eg,
Lemma~\ref{lem:boundun}, the digits of the normalised expansion of any
integer are bounded and contained in a canonical digit set $D=D_U$.  This is
the minimal set of nonnegative digits satisfying $(n)_U ∈ D^*$ for each
$n ∈ ℕ$. 

Conversely, given a finite sequence $g=g_n ⋯ g_0$ on any finite digit set,
we use $[g]_U$ to denote the natural number that has $g$ as a (possibly
non-normalised) expansion in that system, i.e., $[g]_U = ∑_{i=0}^n
g_iu_i$. We will say that $g$ is an expansion of~$n$. Note that the map
$g ↦ [g]_U$ can be applied to any finite sequence over~$ℤ$.  If
$g =g_n⋯ g_0$, we will sometimes use the radix point and write $g\radix$ to
indicate the location of the least significant digit. The radix point is
not needed when the indexing of the digits in $g$ is explicit.  Thorough
descriptions of the properties of addition in these numeration systems can be
found in Frougny's exposition \cite[Chapter~7]{Lothaire}.

\begin{example}[The $k$-bonacci numerations]\label{ex:k-bonacci}
  The most well-known Pisot numerations are the $k$-bonacci numerations
  defined to be the numbers which satisfy the recurrence
  \begin{displaymath}
    u_n= u_{n-1} + u_{n-2} + ⋯ + u_{n-k}.
  \end{displaymath}
  The $k$-bonacci numbers satisfy the recurrence that has
  $X^k-X^{k-1}-⋯-X-1$ as minimal polynomial, which is known to be Pisot
  \cite{Brauer-1951}. Note that, with standard initial conditions, we have
  $u_1=u_0=1$, so the least significant digit for any $(n)_U$ equals~$0$.
  In the $k$-bonacci numeration, each natural number has a unique normalised
  expansion as $n= ∑_{i = 0}^kd_iu_i$ where $d_0=0$, $d_i ∈ \{0,1\}$  and
  $d_id_{i+1}⋯ d_{i+k-1}=0$ for~$i ⩾ 0$. If $k=2$, the numeration $U$ is
  the sequence of Fibonacci numbers $… 5,3,2,1,1\cdot 0$; this numeration
  is also known as the Zeckendorf, or Fibonacci numeration.  If $k=3$, the
  numeration $U = … 14,8,4,2,1,1\cdot 0,0$ is known as the Tribonacci
  numeration.  If $k=4$, $U$ is known as the Tetrabonacci numeration.
\end{example}

\subsubsection{Normalisation for Pisot numerations}

Thorough descriptions of the properties of addition in Pisot numeration
systems can be found in Frougny's exposition
\cite[Chapter~7]{Lothaire}. One important concept on which we base our work
is that of \emph{normalisation}. This consists of converting an arbitrary
expansion of an integer into the normalised expansion. For base-$k$
expansions, $k ∈ ℕ$, this simply translates to addition with carry. For
more exotic numerations, there is often no simple rule. Pisot numeration
systems are somewhat well behaved in this respect.

The following theorem is proved in \cite[Proposition~7.3.11,
Chapter~7]{Lothaire}.  Recall that a set of finite sequences~$K$ is
\emph{regular} if there is a deterministic automaton
$⟨ S,D, \Delta, \{s_0\}, F ⟩$ such that $Δ(s_0, g)$ is in an accepting
state if and only if $g ∈ K$, with $g$ possibly padded with leading zeros
if needed.

\begin{theorem}\label{thm:regularity-0}\cite[Chapter 7]{Lothaire}
  Let $U$ be a Pisot numeration system. For any finite set $B ⊂ ℤ$, the set
  $\{ g ∈ B^*: [g]_U = 0\}$ is regular.
\end{theorem}

\begin{definition} 
  Let $U$ be a numeration defined by a Pisot linear recurrence.  Let
  the finite sequence  $g$ have entries in $ℤ$ and suppose $[g]_U⩾ 0$. Define
  \begin{displaymath}
    \unor(g) ≔ ([g]_U)_U
  \end{displaymath}
  to be the normalised expansion for the natural number $[g]_U$.
\end{definition}

\begin{example}\label{ex:normalisation}
  If $U$ is the Fibonacci numeration, and if $g=1112⋅ $, then
  $[1112\cdot]_U= 1\times 3 +1\times 2 +1 \times 1+2\times 1= 8$ and
  $\unor(1112\cdot ) = (8)_U= 100000\cdot$.

\end{example}

\section{The  $U$-adic integers}\label{sec:U-adic-integers}

In this section, we introduce the space of $U$-adic integers.  This space
is obtained as a quotient group $𝔾/ℍ_U$ where the abelian group~$𝔾$ does
not depend on the numeration~$U$ and $ℍ_U$ is the subgroup of all possible
expansions of~$0$. For a sequence $g = (g_n)_{n ⩾ 0} ∈ ℤ^ℕ$, set
  $\|g\|_∞ ≔ \sup \{ g_n : n ⩾ 0\}$.

Let $𝔾$ be the set of bounded sequences of integers, that is,
\begin{displaymath}
  𝔾 ≔ ℤ^ℕ ∩ ℓ^∞  = \{ g =(g_n)_{n⩾ 0}∈ ℤ^ℕ: \|g\|_∞ < ∞  \}.
\end{displaymath}

The set $𝔾$ is endowed with a commutative group structure, inherited from
$ℤ^ℕ$, of component-wise addition, i.e., $g+g'$ is the sequence defined by 
$(g+g')_n = g_n + g'_n$. The sequence~$g-g'$ is defined similarly.  Each
finite sequence over~$ℤ$ is identified with the element of~$𝔾$ obtained by
appending a tail of zeros. With this operation of component-wise addition,
$𝔾$ is an abelian group. We will not use the natural topology on $𝔾$, which
is that points are close if they agree on a large initial head. Rather,
we introduce the appropriate topology for our setting in
Section~\ref{section:topology}.

We now come to the definition of the subgroup~$ℍ_U$ of~$𝔾$.  It is obtained
first by defining a function $\ord$ on~$𝔾$ and second a
\emph{pseudo-valuation~$\uval$}.  Define the functions
$\ord : 𝔾 → ℕ ∪ \{ ∞\}$ and $\deg : 𝔾 → ℕ ∪ \{ -∞\}$ by
\begin{displaymath}
  \ord(g) ≔ \inf \{ n : g_n ≠ 0\}
  \quad\text{and}\quad
  \deg(g) ≔ \sup \{ n : g_n ≠ 0\}.
\end{displaymath}
with the convention that $\inf ∅ = ∞$ and $\sup ∅ = -∞$ and thus
$\ord(\period{0}) =∞$ and $\deg(\period{0})= -∞$.  The function~$\ord$
satisfies $\ord(g+g') ⩾ \min \{ \ord(g), \ord(g')\}$.  In
Section~\ref{sec:notation}, we extend the functions $\ord$ and~$\deg$ to
two-sided infinite sequences.

Note that $𝔾$ and~$\ord$ are universally defined and do not depend on any
numeration.  We now define the subgroup~$ℍ_U$ by introducing the
numeration~$U$. We will see in Section~\ref{section:topology} that this
numeration also defines the appropriate topology on~$𝔾$, turning it into a
topological group.  Define the function $\uval : ℤ → ℕ ∪ \{ ∞\}$ by
\begin{displaymath}
  \uval(n) ≔ \ord((|n|)_U).
\end{displaymath}
Note that with this definition, $\uval(0)=\infty$.  The function~$\uval$
does not enjoy all the nice properties of the $p$-valuation; in particular
it is not non-Archimedean.  Note that, while we have defined $\uval$ on the
integers, its definition can be extended to act on finite sequences $g$, namely,
$\uval(g) ≔ \uval([g]_U)$.

\begin{definition}\label{def:preserve-zeros}
  Let $U$ be a Pisot numeration. We say that $U$ \emph{preserves zeros} if
  for each natural~$c$, there exists a constant~$K_c = K_c(U)$ such that
  for each finite  sequence $g ∈ ℤ^*$ with $\|g\|_∞ ⩽ c$, we have
  \begin{displaymath}
    \uval(g)  ⩾ \ord(g)-K_c.
  \end{displaymath}
\end{definition}
Definition~\eqref{def:preserve-zeros} can be relaxed. In particular, it can
be shown that if such a constant~$K_c$ exists for large enough~$c$, then it
exists for all~$c$.  More precisely let $d = \max D$ where $D$ is the
canonical digit set. If $K_c$ exists for $c = 2d$, then $K_c$ exists for
all~$c$.

\begin{example}
  If the numeration $U = (u_n)_{n ⩾ 0}$ is given by $u = p^n$ for some
  possibly non-prime integer $p ⩾ 2$, then $U$ preserves zeros, and $K_c=0$
  for each $c ∈ ℕ$.
\end{example}

\begin{example}\label{ex:k-bonacci-preserve-zeros}
  We will prove in Section~\ref{sec:condition-F} that a Pisot numeration
  preserves zeros if and only if the corresponding Pisot number satisfies
  the  Condition~F as defined in~\cite{Frougny-Solomyak-1992}. Since each
  $k$-bonacci polynomial satisfies the condition of \cite[Theorem
  2]{Frougny-Solomyak-1992}, it satisfies the Condition~F and hence the
  $k$-bonacci numeration $U$ preserves zeros by
  Proposition~\ref{pro:bnorm2unorm}.

  The constants $K_c$ can be computed using transducers
    performing normalisation.  For example, if $U$ is the Fibonacci
  numeration, then it can be shown that $K_1 = 2$ and if $c = 2$, then
  $K_2⩾ 3$, because, for example, the normalisation of
  $2\bar{2}2\bar{2}20000$ is $100101010$. Computation suggests that
  $K_2=3$.
\end{example}
The normalised expansion of $m+n$ can be obtained by first adding
component-wise the normalised expansions of $m$ and~$n$, and then
normalising the resulting finite sequence which is over the alphabet
$\{0, 1, … ,2\max D \}$. Corollary~\ref{cor:zvalsum} follows
from this. Note that the $p$-adic valuation satisfies a similar relation
without the constant on the right hand side. The following corollary justifies our description of $\uval$ as a  pseudo-valuation.
\begin{corollary} \label{cor:zvalsum}
  Let $U$ be a numeration that preserves zeros, with finite canonical digit
  set~$D$.  Let $d ≔ \max D$.  Then for $m, n ∈ ℤ$, we have
    \begin{displaymath}
    \uval(m+n) ⩾ \min(\uval(m), \uval(n))-K_{2d}.
  \end{displaymath}
\end{corollary}
\begin{proof}
  Let $u = (|m|)_U$ and $v = (|n|)_U$ be the $U$-expansions of $|m|$
  and~$|n|$.  By definition $\|u\|_∞, \|v\|_∞ ⩽ d$.  If the signs of $m$
  and~$n$ are equal, then $|m+n| = |m|+|n|$.  If the signs of $m$ and~$n$
  are different, then $|m+n|$ is equal to either $|m|-|n|$ or $|n|-|m|$
  depending on the sign of $m-n$.  This means that in all cases $|m+n|$
  belongs to $\{[u+v]_U, [u-v]_U, [v-u]_U\}$.  Since
  $\|u+v\|_∞, \|u-v\|_∞, \|v-u\|_∞ ⩽ 2d$, the result is proved.
\end{proof}

\begin{lemma}\label{lem:altdefH}
  Let $U$ be a numeration that preserves zeros and let $g ∈ 𝔾$.  Then,
  \begin{displaymath}
    \lim_{n → ∞} \uval(g_{[0:n)}) = ∞ \iff \limsup_{n → ∞}\uval(g_{[0:n)}) = ∞
  \end{displaymath}
\end{lemma}
\begin{proof}
  One implication is clear. To prove the other, we start with the following
  remark.  By definition of $𝔾$, if $g ∈ 𝔾$ then there exists a
    constant~$c$ such that $|g_n| ⩽ c$ for each $n$.
  Since~$U$ preserves zeros, there exists an integer~$K_c$ such that
  $\uval(w) ⩾ \ord(w)-K_c$ for each factor~$w$ of~$g$.  We claim that if
  $n⩾ m ⩾ 0$, then
        \begin{displaymath}
  		\uval(g_{[0:n)})⩾ \min(m, \uval(g_{[0:m)})-K_c-K_{2d}.
  \end{displaymath}
  Let $M= \min(m, \uval(g_{[0:m)}))$ and let $u = g_{[0:m)} $,
  $v = g_{[0:n)}$ and $w = v - u$, where the symbol~$-$ denotes the
  component-wise subtraction of finite sequences.  Then by definition of $M$,
  $\uval(g_{[0:m)})⩾ M ⩾ M-K_c$, and the finite sequence~$w$ satisfies
  $\ord(w) ⩾ m ⩾ M$, so that $\uval(w) ⩾ M-K_c$.  It follows 
  from the relation $[v]_U = [u]_U+[w]_U$ and Corollary~\ref{cor:zvalsum} that
  $\uval(v) ⩾ M-K_c-K_{2d}$.  This completes the proof of the claim.
  
  Now suppose that $\limsup_{n → ∞} \uval(g_{[0:n)}) = ∞$, and let $L$ be
  some fixed integer~$L ⩾ 0$. By hypothesis we can choose an integer~$m$
  such that $m ⩾ K_c+L+K_{2d}$ and $\uval(g_{[0:m)}) ⩾ K_c+K_{2d}+L$, and now
  the proved claim implies that if $n ⩾ m$, then $\uval(g_{[0:n)}) ⩾
  L$. But this is true for any L and therefore
  $\limsup_{n → ∞} \uval(g_{[0:n)}) = \lim_{n → ∞} \uval(g_{[0:n)})$,
  proving the lemma.
\end{proof}
We define 
\begin{displaymath}
  ℍ_U ≔ \{ g ∈ 𝔾 : \lim_{n → ∞} \uval(g_{[0:n)}) = ∞ \}.
\end{displaymath}
By Lemma~\ref{lem:altdefH}, the limit occurring in this definition could be
replaced by a $\limsup$.  It is straightforward that $ℍ_U$ is a subgroup
of~$𝔾$.
\begin{definition}
  We call  the group $\uspace ≔ 𝔾/ℍ_U$ the \emph{U-adic integers}.
\end{definition}

It does not seem to be easy to use $\uval$ to define a topology
on~$\uspace$. For Pisot numerations we will define, in
Section~\ref{section:topology}, a norm on~$𝔾$, through which we can endow
$\uspace$ with the quotient topology.

\subsection{Normalised elements in $\uspace$}
\label{sec:beta-numeration}

In this section, we show that each class of $\uspace ≔ 𝔾/ℍ_U$ contains
a $U$-normalised sequence.

\begin{notation*}
  Given a numeration $U$, set
  \begin{align}\label{eq:closure}
    ℕ_U & ≔ \{ (n)_U: n∈ ℕ\} \nonumber \\
    \text{and}\qquad
    \overline{ℕ}_U & ≔ \{ g ∈ 𝔾 : g[0:n) ∈ ℕ_U \text{ for each $n ⩾ 0$}\}.
  \end{align}
  We use this notation because $\overline{ℕ}_U$ is the topological closure
  of~$ℕ_U$, not just in the natural topology for~$𝔾$, but also for the
  topology that we introduce in Section~\ref{section:topology}. We say that
  $ g ∈ 𝔾$ is \emph{$U$-normalised} if $g ∈ \overline{ℕ}_U$.
\end{notation*}
Note that $\overline{ℕ}_U$ is \emph{not} invariant under shifting. For
example, if $U$ is the Fibonacci numeration with standard initial
conditions, then normalised sequences always start with~$0$, and shifting
can destroy this.

In this section we will show that, for numerations arising from a Pisot
numeration that preserves zeros, we can always find representatives of
elements in $\uspace$ that are $U$-normalised. For this we recall some basic
notions around $β$-numerations defined by a leading root of a Pisot
polynomial~$P$.

Let $β > 1$ be a real number.  A \emph{$β$-expansion} of a non-negative
real number~$x$ is a sequence $(d_k)_{k ⩽ n}$ of integers such that
$x = ∑_{k ⩽ n} d_kβ^k$.  Note that the integer~$n$ might be negative if
$x< 1$.  The \emph{normalised $β$-expansion} is the lexicographically
greatest $β$-expansion.  It is obtained by the greedy algorithm and it is
characterized by the fact that
\begin{equation}\label{eq:beta-normalised}
  ∑_{i ⩽ k} d_iβ^i < β^{k+1}
\end{equation}
for each $k ⩽ n$.  A sequence satisfying
Condition~\eqref{eq:beta-normalised} is called \emph{$β$-normalised}.

A related expansion for $x ∈ [0,1]$ is the Rényi expansion, defined using
the map $T_β(x) = β x \bmod 1$. If $0 ⩽ x<1$ and $x = ∑_{k⩽ n} d_kβ^k$,
then $d_k$ equals the integer part of $β T_β^{k-1} (x)$.  The $β$-expansion
of $x = 1$, which is just $β^0$, differs from its Rényi expansion , which
we write as $d_β(1) = t_1t_2t_3⋯$.  Note that we use $t_k$ and not $d_k$
for $d_β(1)$ due to its ubiquitous use.

Denote by $D_β$ the set of $β$-expansions of $x ∈ [0,1)$, and denote the
shift-orbit closure of $D_β$ by~$S_β$. Let $σ$ denote the (one-sided) shift
map.  Let $\ltlex$ denote the lexicographical order.  If $d_β(1)$ is
infinite, set $d_β^{*}(1) ≔ d_β(1)$.  Otherwise if $d_β(1)=t_1 ⋯ t_m$ with
$t_m ≠ 0$, set $d_β^{*}(1) ≔ \period{t_1⋯ t_{m-1}(t_m-1)}$.  Note that all
suffixes of $d_β(1)$ are lexicographically smaller than $d_β(1)$
\cite[Chapter 7]{Lothaire}.

\begin{example}[The $k$-bonacci numerations]\label{ex:k-bonacci-2}
  Recall the numerations from Example~\ref{ex:k-bonacci}, which satisfy the
  recurrence that has $X^k-X^{k-1}-⋯-X-1$ as minimal polynomial. In this
  case $d_β(1)= \underbrace{1⋯1}_{k}$.
\end{example}

 We record the following fundamental result
 \cite{Parry-1960,Bertrand-1977,Schmidt-1980}.

\begin{theorem}\label{thm:Pisot-properties}
  We have $x∈ S_β$ if and only if for each $p⩾ 0$,
  $σ^p(x) \ltlex d_β^*(1)$.  If $β$ is Pisot, then $d_β^*(1)$ is eventually
  periodic and $S_β$ is sofic.
 \end{theorem}

The following lemma is classical.  It tells us that the $n$-th term of a
Pisot numeration grows like $β^n$ with an exponentially decaying
difference.  We include the proof for completeness.
\begin{lemma} \label{lem:boundun}
  Let $U = (u_n)_{n ⩾ 0}$ be a Pisot numeration associated to~$β$. Then
  there are real numbers $0<\alphanorm<1$,  $λ>0$ and $K>0$ such that
  $|u_n - λβ^n| ⩽ K\alphanorm^n$ for each $n ⩾ 0$.
\end{lemma}
\begin{proof}
  We  let $α_0 = β$ and let $α_1,…,α_{d-1}$ be the algebraic
  conjugates of~$β$.  Since the annihilating  polynomial of $β$ is irreducible, a
  general Binet formula tells us that there exist coefficients
  $(λ_k)_{0 ⩽ k ⩽ d-1}$ such that $u_n = ∑_{k=0}^{d-1}λ_kα_k^n$ for each
  $n ⩾ 0$. Thus setting $\alphanorm ≔ \max_k \{ |α_k| : 1 ⩽ k⩽ d-1\}<1$ and
  $K = ∑_{k=1}^{d-1}|λ_k|$,
  \begin{displaymath}
    \left|u_n - λ_0β^n\right| ⩽
    ∑_{k=1}^{d-1} |λ_kα_k^n| ⩽ ∑_{k=1}^{d-1} |λ_k|\alphanorm^n  ⩽ K\alphanorm^n 
  \end{displaymath}
   Note that
  $λ ≔ λ_0 >0$ since $β>1$ and the $u_n$
    are eventually positive.
\end{proof}
 
We prove in Section~\ref{sec:condition-F} that, for a numeration~$U$
associated to a Pisot number~$β$, the preservation of zeros by~$U$ is
equivalent to the Condition~F for~$β$, introduced by Frougny and Solomyak
\cite{Frougny-Solomyak-1992}, which implies that $d_β(1)$ is finite.  Below
we give a shorter and direct proof that if $U$ preserves zeros, then
$d_β(1)$ is finite.
\begin{lemma}\label{lem:dbetafinite}
  Let $U$ be a Pisot numeration associated to~$β$.  If
  $U$ preserves zeros, then $d_β(1)$ is finite.
\end{lemma}
\begin{proof}
  If $d_β(1)$ is not finite, then by Theorem~\ref{thm:Pisot-properties}, it
  is ultimately periodic; set $d_β(1) = t_1 ⋯ t_ℓ \period{t_{ℓ+1} ⋯ t_m}$.
  Consider the finite sequence $g = \period{0}1(-t_1) ⋯ (-t_ℓ)0^k$.
  Using Lemma~\ref{lem:boundun}, it is easily proved that the most
  significant digits of the normalisation of~$g$ are $(t_{ℓ+1} ⋯ t_m)^p$
  where $p$ has magnitude $k/(m-ℓ)$.  This shows that $U$ does not preserve
  zeros.
\end{proof}

\begin{lemma}\label{lem:positive-H-entries}
    Let $U=(u_n)_{n⩾ 0}$ be a Pisot numeration that preserves zeros.  Then there
    exists $h=(h_n) ∈ ℍ_U$ with $h_n > 0$ for each~$n ⩾ 0$.
\end{lemma}
 \begin{proof} Let $P$ be the minimal polynomial of the
  recurrence defining $U$ and let $β$ be the leading root.
  Lemma~\ref{lem:dbetafinite} tells us that $d_β(1) = t_1⋯ t_m$ is finite,
  where $t_1 > 0$ and $t_m > 0$. Necessarily,
  $u_{n+m} = t_1u_{n+m-1} + ⋯ +t_mu_n$.  If $n=qm+r$ where $q ⩾ 1$ and
  $0 ⩽ r < m$ then $(u_n)_U = 10^{qm+r}$, so we also have
  \begin{displaymath}
   [(t_1⋯ t_{m-1}(t_m -1))^{q-1} t_1⋯ t_m0^r]_U = u_n.
  \end{displaymath}
  This implies that
  $h^{(r)}= \period{t_1⋯ t_{m-1}(t_m -1)}t_1⋯ t_m0^r ∈ ℍ_U$ for each
  $r ⩾ 0$, and so also $h ≔ ∑_{r = 0}^{m-1} h^{(r)} ∈ ℍ_U$. Furthermore,
  $h$ has strictly positive entries, as each entry~$h_i$ satisfies
  $h_i ⩾ t_m > 0$ if $0 ⩽ i ⩽ m-1$ and $h_i ⩾ t_1 > 0$ if $m ⩽ i$.
\end{proof}

We illustrate the statement of Lemma~\ref{lem:positive-H-entries} with
the following example.
\begin{example} 
  If $U$ is the n-bonacci numeration, then following the method in the
  proof of Lemma~\ref{lem:positive-H-entries}, we have that
  $\period{n-1}n ∈ ℍ_U$.
\end{example}

The next proposition tells us that each element of~$\uspace$, i.e., each
class of $𝔾/ℍ_U$, contains at least one $U$-normalised element.
 
\begin{proposition}\label{pro:trace}
  Let $U$ be a numeration that preserves zeros.  For each $g ∈ 𝔾$ there
  exists $g' ∈ \overline{ℕ}_U$ such that $g-g'∈ ℍ_U$.
\end{proposition}
\begin{proof} 
  By Lemma~\ref{lem:positive-H-entries}, there exists $h ∈ ℍ_U$ with
  strictly positive entries. By replacing $g$ with $g+kh$ for some $k$, we
  can assume that each entry of $g$ is at least $\max D$, where $D$ is the
  canonical digit set.

  Since for each $n$, $\unor(g_{[0,n]}) $ is defined over the finite set
  $D$, we take $g'$ to be a limit point of the sequence
  $(\unor(g_{[0,n]}))_n$. We claim that $g' ∈ \overline{ℕ}_U$. Indeed for
  each $k$, $g'_{[0,k]}$ is the start of $\unor(g_{[0,n_k]})$ for some
  $n_k$ large enough. The claim now follows as the head of a sequence
  in~$ℕ_U$ belongs to~$ℕ_U$.

  It remains to show that $g-g'∈ ℍ_U$. Given $ℓ$ we need to show that there
  is $N_ℓ$ such that $\ord(\unor((g-g')_{[0,n]})) > ℓ$ if $n ⩾ N_ℓ$. If
  $\|g-g'\|_∞ ⩽ c$, then since zeros are preserved, there exists $K_c$
  such that $\ord(\unor(w))>\ord(w)-K_c $ for a finite sequence
  $w= w_0 ⋯ w_n$ with $|w_k|⩽ c$. There exists $N_ℓ$ such that
  $ \unor(g_{[0,N_ℓ]}) $ agrees with $g'$ on the indices $[0,ℓ+K_c]$.  Note
  that since the entries of $g$ are at least $\max D$, and since each entry
  of~$g'$ is at most $\max D$, $g-g'$ has non-negative entries.  This, with
  the fact that
  \begin{align*}
    [(g-g')_{[0,N_ℓ]}]_U & = [g_{[0,N_ℓ]}]_U - [g'_{[0,N_ℓ]}]_U \\
                       & = [\unor(g_{[0,N_ℓ]})]_U - [g'_{[0,N_ℓ]}]_U \\
                       & = [\unor(g_{[0,N_ℓ]})-g'_{[0,N_ℓ]}]_U
  \end{align*}
  implies that
  $\unor[(g-g')_{[0,N_ℓ]}]_U=\unor[\unor(g_{[0,N_ℓ]})-g'_{[0,N_ℓ]}]_U$ starts
  with $ℓ$ zeros.
\end{proof}

\begin{example}
  If the numeration $U = (u_n)_{n ⩾ 0}$ is given by $u = p^n$ for some
  possibly non-prime integer $p ⩾ 2$, then $\uspace = ℤ_p$, the additive
  group of the $p$-adic integers.  Indeed, it is clear that any sequence
  over the alphabet $\{0, …, p-1\}$ is normalised, and two normalised
  sequences $g$ and~$g'$ satisfy $g'-g ∈ ℍ_U$ if and only if $g = g'$.
\end{example}

With Proposition~\ref{pro:trace}, we can extend the notion of normalisation to
elements of~$𝔾$. Namely, we can think of the canonical surjection
\begin{equation}\label{eq:canonical-factor}
  π_{ℍ_U} : 𝔾 →  𝔾/ℍ_U = \uspace
\end{equation}
as being a \emph{normalisation} map: the normalisation of $g$ is the set of
elements in the equivalence class $π_{ℍ_U}(g)$ which are
$U$-normalised. Furthermore $π_{ℍ_U}(\overline{ℕ}_U) =\uspace$.

The following theorem tells us that there are many classes which contain
only one $U$-normalised element.  In particular, if a $U$-normalised
sequence~$g$ has its non-zero entries sufficiently far apart, then $g$ is
the unique $U$-normalised sequence in its class.  
\begin{theorem} \label{thm:single}
  Let $U$ be a Pisot numeration that preserves zeros. Then there exists a
  constant~$K$ with the following property.  If $g = (g_n)_{n ⩾ 0}$ is a
  $U$-normalised sequence such that $g_{[n,n+K)} = 0^K$ for infinitely
  many~$n$, then if $g'$ is $U$-normalised and $g' - g ∈ ℍ_U$ then
  $g' = g$.
\end{theorem}
  
In the case of the Fibonacci numeration~$U$, each $U$-normalised sequence
with infinitely many occurrences of the block~$00$ is the unique
$U$-normalised sequence in its class. Said differently, if $g$ and~$g'$ are
two different $U$-normalised sequences in the same class, then both $g$
and~$g'$ end with a tail $(01)^ℕ$.

In Lemmas~\ref{lem:valprefix} to~\ref{lem:uvaldiff}, we will use constants
$K_c$ given by Definition~\ref{def:preserve-zeros} of zero preservation.
Recall that if the finite sequence~$w$ satisfies $\|w\|_∞ ⩽ c$, then
$\uval(w) ⩾ \ord(w) - K_c$.  The following lemma states that the valuation
of the heads of a sequence in~$ℍ_U$ grows like the length of these heads.
\begin{lemma} \label{lem:valprefix}
  Let $U$ be a Pisot numeration that preserves zeros, let $D$ be the
  canonical digit set and let $d^* ≔ \max D$. For each~$c$, if $h ∈ ℍ_U$
  and $\|h\|_∞ ⩽ c$, then $\uval([h_{[0,n)}]_U) ⩾ n - K_{c+d^*}$ for each
  integer $n ⩾ 0$.
\end{lemma}
\begin{proof}
  Since $U$ preserves zeros, if the finite sequence~$w$ satisfies
  $\|w\|_∞ ⩽ c+d^*$, then $\uval([w]_U) ⩾ \ord(w) - K_{c+d^*}$.  Since $h$
  belongs to~$ℍ_U$ there exists an integer $m ⩾ n$ such that
  $\uval([h_{[0,m)}]_U) ⩾ n$.  Let us set $v = \unor([h_{[0,m)}]_U)$ and
  $w = h_{[n,m)}0^n\cdot$; since $[h_{[0,n)}]_U = [v-w]_U$, $\ord(v) ⩾ n$,
  $\ord(w) ⩾ n$ and $\|v-w\|_∞ ⩽ c+d^*$, one has
  $\uval([h_{[0,n)}]_U) = \uval([v-w]_U) ⩾ n - K_{c+d^*}$.
\end{proof}

The following lemma states that if a sequence $h$ in~$ℍ_U$ has a long
enough block of zeros, the value of the head of~$h$ before that block
is~$0$.
\begin{lemma} \label{lem:boundzb}
  Let $U$ be a Pisot numeration that preserves zeros. For each~$c$, there
  exists~$L_c$ such that if $h ∈ ℍ_U$, $\|h\|_∞ ⩽ c$ and
  $h_{[n,n+L_c)} = 0^{L_c}$ for some integer $n ⩾ 0$, then
  $[h_{[0,n)}]_U = 0$.
\end{lemma}
\begin{proof}
  For each constant~$c$, there exists a constant~$C$ such that, for each
  finite sequence~$g$, $\|g\|_∞ ⩽ c$ implies $|\unor(g)| ⩽ |g|+C$. Let $D$
  be the canonical digit set and let $d^* ≔ \max D$.  We claim that the
  constant $L_c ≔ K_{c+d^*} + C$ satisfies the required property.  Suppose
  that $h ∈ ℍ_U$, $\|h\|_∞ ⩽ c$ and $h_{[n,n+L_c)} = 0^{L_c}$.  Suppose
  also by contradiction that $[h_{[0,n)}]_U ≠ 0$.  Let $g$ be the finite
  sequence $\unor(h_{[0,n)})$.  Since $[g]_U ≠ 0$ and $|g| ⩽ n+C$, then
  $\ord(g) < n+C$.  Since $h_{[n,n+L_c) }= 0^{L_c}$,
  $\unor(h_{[0,n+L_c)}) = g$ and thus
  $\uval([h_{[0,n+L_c)}]_U) < n+C = n+L_c-K_{c+d^*}$, which is a
  contradiction to the statement of Lemma~\ref{lem:valprefix}.
\end{proof}

\begin{lemma} \label{lem:uvaldiff}
  For a Pisot numeration $U$ that preserves zeros, let $D$ be the canonical
  digit set and let $d^* ≔ \max D$.  Let $g$ and~$g'$ in~$𝔾$ be
  $U$-normalised sequences, and let $d' ≔ \deg(g')$ and $d ≔ \deg(g)$. If
  $d'> d$ and $[g'_d ⋯ g'_0]_U ≠ [g]_U$, then $\uval(g'-g) < d + K_{2d^*}$.
\end{lemma}
Note that the condition $d' > d$ implies $[g']_U > [g]_U$ and
thus $[g'-g]_U > 0$. 

\begin{proof}
  Let $g'' ≔ g'_d ⋯ g'_0$.  We distinguish two cases depending on whether
  $[g'']_U > [g]_U$ or $[g'']_U < [g]_U$.  Note that equality is excluded
  by hypothesis.

  In the first case $[g'']_U > [g]_U$, we claim that the stronger
  inequality $\uval(g'-g) ⩽ d$ holds.  Let $w$ be the sequence
  $\unor(g''-g)$; it has length at most $d$. If needed, we pad $w$ with
  final zeros so that it has length $d$. It is straightforward to check
  that $\unor(g'-g)= g'_{d'} ⋯ g'_{d+1}w$, and the claim follows.

  We now consider the second case $[g'']_U < [g]_U$.  Let
  $k ≔ d+1+\ord(g'_{d'}⋯ g'_{d+1})$.  By definition of~$k$, $g'_k ≠ 0$ and
  $g'_i = 0$ for $d+1 ⩽ i < k$.  We first suppose that $k ⩽ d + K_{2d^*}$.
  Now $[g'_{[0:k]}-g]_U > 0$ since $g$ is normalised.  Let
  $w ≔ \unor(g'_{[0:k]}-g)$; the word~$w$ satisfies
  $|w| ⩽ k+1 ⩽ d +1 + K_{2d^*}$.  It follows that
  $\uval(w) ⩽ d + K_{2d^*}$.  Since $g'$ was $U$-normalised,
  $\uval(g'-g) ⩽ d + K_{2d^*}$ because $\unor(g'-g)_{[0:|w|)} = w$.
  
  Finally suppose that $k > d + K_{2d^*}$.  This means that
  $\uval(g'_{d'} \cdots g'_{ d+1} 0^{d+1}) > d + K_{2d^*}$.  We have
  $[g-g'']_U = [g'_{d'}\cdots g'_{d+1}0^{d+1}]_U - [\unor(g'-g)]_U$. Now
  $\uval (g-g'') ⩽ d.$ Suppose by contradiction that
  $\uval(g'-g) ⩾ d + K_{2d^*}$. Then we have written $[g-g'']_U$ as a
  difference of two elements, each of whose valuation is strictly greater
  than $d+K_{2d^*}$. By the definition of $K_{2d^*}$, we have
  $\uval (g-g'') > d+K_{2d^*}-K_{2d^*} = d$, and this is a contradiction.
\end{proof}

\begin{proof}[Proof of Theorem~\ref{thm:single}]
  As before, let $D$ be the canonical digit set and let $d^* ≔ \max D$.  We
  claim that the constant
  $K = 3\max(K_{3d^*},L_{2d^*},K_{2d^*}) = 3\max(K_{3d^*},L_{2d^*})$, where
  $L_{2d^*}$ is the constant provided by Lemma~\ref{lem:boundzb},
  satisfies the required property.
  
  We consider a $U$-normalised sequence $g = (g_n)_{n ⩾ 0}$ satisfying the
  hypothesis of the theorem.  Suppose that $g' = (g'_n)_{n ⩾ 0}$ is a
  $U$-normalised sequence such that $h = g' - g ∈ ℍ_U$.  Since $g$ and $g'$
  are $U$-normalised, the sequence~$h$ satisfies $\|h\|_∞ ⩽ 2d^*$.  The
  statement of Theorem~\ref{thm:single} holds trivially if the sequence~$g$
  has finitely many non-zero entries.  Indeed, since both $g$ and~$g'$ are
  $U$-normalised, $g = g' = (n)_U$ where $n = [g]_U$.  Therefore, we can
  assume that $g_n ≠ 0$ infinitely often. The statement also holds
  trivially if $[g'_{[0:n)}]_U = [g_{[0:n)}]_U$ holds for infinitely
  many~$n$ since $g$ and~$g'$ are $U$-normalised.
  
  Therefore, we can assume that $[g'_{[0:n)}]_U = [g_{[0:n)}]_U$ holds for
  finitely many integers~$n$.  Lemma~\ref{lem:boundzb} now implies that the
  sequence~$h$ has finitely many blocks of zeros of length at
  least~$L_{2d^*}$.  Hence there exists a large enough integer~$N$ such
  that if $[g'_{[0:n)}]_U = [g_{[0:n)}]_U$ then $n < N$.  By hypothesis,
  there exists an integer~$n$ such that $n ⩾ N$, $g_{n-1} ≠ 0$ and
  $g_{[n,n+K)} = 0^K$.  By definition of~$N$, all occurrences
  of~$0^{L_{2d^*}}$ in~$h$ are contained in its prefix~$h_{[0:N)}$.  Since
  $K ⩾ 3L_{2d^*}$, there exists an integer~$ℓ$ such that $h_ℓ = g'_ℓ ≠ 0$
  and $n+2K/3 ⩽ ℓ < n+K$. By Lemma~\ref{lem:uvaldiff},
  $\uval(h_{[0:ℓ]}) = \uval(g'_{[0:ℓ]}-g_{[0:n)}) ⩽ n + K_{2d^*}$. By
  choice of $K$ and~$ℓ$, we have $ℓ -n ⩾ 2K/3 ⩾ 2K_{3d^*}$, and this
  implies that
  \begin{displaymath}
    \uval(h_{[0:ℓ]}) ⩽  n + K_{2d^*} ⩽  n + K_{3d^*} ⩽ ℓ - K_{3d^*},
  \end{displaymath}
  which contradicts Lemma~\ref{lem:valprefix}.
\end{proof}

We end this section with an example of a description of representations of
negative integers in $\uspace$.  Let $g$ be a finite sequence over $ℤ$ such
that $[g]_U ⩾ 0$.  It is easily verified that $g$ and $g' ≔ \unor(g)$ are
in the same class.  Indeed $h = g'-g$ belongs to $ℍ_U$ since
$[(g' - g)_{[0:n)}]_U = 0$ for $n ⩾ \max(|g|,|\unor(g)|)$.  It follows that
$ℕ_U$ is isomorphic to~$ℕ$ for addition and also subtraction when defined.

\begin{example}
  We describe the negative integers in~$\uspace$ when $U$ is the Tribonacci
  numeration.  Each negative integer has three distinct expansions
  in~$\overline{ℕ}_U$ which can be obtained as follows.  Suppose that $n$
  is fixed. For $k$ large enough, the most significant digits of the
  normalised expansion of $u_k-n$ are $110110⋯$.  One expansion of~$-n$ is
  then $\period{110}(u_k-n)_U$.  For $k$ and $k'$ not congruent modulo~$3$,
  one gets different expansions.  We give below the expansions for
  $n = -1,-2,-3$.
  \begin{displaymath}
    \begin{array}{|rr||rr||rr|} \hline
      \rule{0pt}{1.1\normalbaselineskip} 
      -1 & \period{110}0 & -2 & \period{110}00  & -3 & \period{110}010 \\
         & \period{101}0 &    & \period{101}00  &    & \period{101}000 \\
         & \period{110}  &    & \period{101}010 &    & \period{101}0110 \\
      \hline
    \end{array}
  \end{displaymath}
\end{example}

\section{From $\uspace$ to a  torus}
\label{sec:IsomorphisWithTorus}

So far we have just considered $𝔾$ and $\uspace$ as groups, but we can
endow each of them with a topology, as we describe in
Section~\ref{section:topology}.  In this section we show that there is a
continuous homomorphism from~$\uspace$ to a torus. We then show that if $β$
is unimodular, then $\uspace$ is continuously isomorphic to a torus.
   
Let $U$ be a Pisot numeration of degree $d$ whose minimal polynomial 
 has $ℓ+1$ real roots and $2k$ complex roots.  Denoting the leading root of
the  polynomial by~$β$, we label its algebraic conjugates
\begin{equation}\label{eq:conjugates}
  α_1,… ,α_{2k},α_{2k+1},…,α_{d-1}.
\end{equation}
The $2k$ complex Galois conjugates are $α_1,…,α_{2k}$ where
$α_{k+i} = \overline{α}_i$ for $1 ⩽ i ⩽ k$ while there are $ℓ ≔ d-1-2k$ 
real conjugates $α_{2k+1},…,α_{d-1}$. When needed, we will set $α_0 = β$.

We consider $ℂ^k × ℝ^ℓ$ as an $ℝ$-algebra with component-wise operations.
Let us define the vector~$\alphavec ∈ ℂ^k × ℝ^ℓ$ by
\begin{equation}\label{eq:norm}
  \alphavec ≔ (α_{k+1}, … ,α_{d-1})
  \quad\text{and}\quad
  \alphanorm ≔ \|\alphavec\|_∞;
\end{equation}
note that $\alphanorm < 1$ because $β$ is assumed to be a Pisot number.
Note that the conjugates $α_1,…, α_k$ have been removed.  By definition,
$\alphavec^i = (α_{k+1}^i, … ,α_{d-1}^i )$.

\begin{example}
  Recall the $k$-bonacci numerations from Example~\ref{ex:k-bonacci}.  If
  $U$ is the Fibonacci numeration, then $k = 0$, $ℓ = 1$ and the vector
  $\alphavec$ is the vector $(α_1) ∈ ℝ$ where $α_1 = (1-\sqrt{5})/2$, the
  conjugate of the golden ratio.
    
  If $U$ is the Tribonacci numeration, then $k = 1$, $ℓ = 0$ and the
  vector~$\alphavec$ is the vector $(α_2) ∈ ℂ$ where
  $α_0 = β,α_1,α_2 = \overline{α}_1$ are the roots of the polynomial
  $X^3-X^2-X-1$.
    
  If $U$ is the Tetrabonacci numeration, then $k = ℓ = 1$ and the
  vector~$\alphavec$ is the vector $(α_2,α_3) ∈ ℂ × ℝ$ where
  $α_0 = β, α_1, α_2 = \overline{α}_1, α_3$ are the roots of the polynomial
  $X^4-X^3-X^2-X-1$.
\end{example}

Define the group homomorphism $\rval: 𝔾→ ℂ^k× ℝ^ℓ$ by
\begin{equation} \label{eq:def-Phi}
  \rval((g_n)_{n⩾ 0}) ≔ ∑_{n ⩾ 0}g_n \alphavec^{n+d-1} .
\end{equation}
Note that $\rval$ is well defined as $\alphanorm < 1$ and the entries in
$(g_n)_{n ⩾ 0}$ are bounded.  Moreover, we have the following trivial
bound:
\begin{equation}\label{eqn:rvalbound}
  \|\rval(g)\|_∞ ⩽ \|g\|_∞\alphanorm^{\ord(g)+d-1} /(1-\alphanorm).
\end{equation}
Define the lattice $𝕃 ⊂ ℂ^k× ℝ^ℓ$ to be the set of $ℤ$-linear combinations of the vectors $\alphavec^i$ for $0⩽ i⩽ d-2 $, i.e., 
\begin{displaymath}
  𝕃 ≔   \{k_0\onevec+ k_1\alphavec + \cdots + k_{d-2}\alphavec^{d-2} : k_i\in ℤ\} ,
\end{displaymath}
where $\onevec$ is the vector $\alphavec^0 = (1,…,1) ∈ ℂ^k× ℝ^ℓ$.  The
following lemma states that the lattice~$𝕃$ has the same dimension as
$ℂ^k× ℝ^ℓ$, which ensures that the points in $𝕃$ are isolated.
\begin{lemma}\label{lem:torus}
  The $ℝ$-vector space spanned by~$𝕃$ is $ℂ^k× ℝ^ℓ$ and the quotient
  $(ℂ^k× ℝ^ℓ)/𝕃$ is thus isomorphic to the torus
  $ℝ^{2k+ℓ}/ℤ^{2k+ℓ} = 𝕋^{2k+ℓ}$.
\end{lemma}
Note that it is important to consider $ℝ$-vector spaces as we want $ℂ^k$ to
be isomorphic to $ℝ^{2k}$.

\begin{proof}
  It suffices to prove that the $d-1$ vectors
  $\onevec, \alphavec, …, \alphavec^{d-2}$ are $ℝ$-linearly independent.
  Suppose by contradiction that there are $d-1$ real numbers
  $r_0, …, r_{d-2} ∈ ℝ$ such that $∑_{i=0}^{d-2}r_i\alphavec^i = 0$.  This
  equality means that the $d-k$ conjugates $α_{k+1},…,α_{d-1}$ are roots of
  the polynomial $Q(X) = ∑_{i=0}^{d-2}r_iX^i$.  Since the coefficients
  of~$Q(X)$ are real numbers, the conjugates $α_1,…,α_k$ are also roots
  of~$Q(X)$.  Since $Q(X)$ has degree at most $d-2$ and has at least $d-1$
  roots, each $r_i=0$ for $0⩽ i⩽ d-2$.
\end{proof}

The following proposition shows that the function $\rval: 𝔾→ ℂ^k× ℝ^ℓ$
induces a function from $\uspace$ to $(ℂ^k× ℝ^ℓ)/𝕃$. 
It is in this proposition that we first need  the initial  conditions
  of the numeration $U$ to be standard.

\begin{proposition} \label{pro:kernel}
  Let $U=(u_n)_{n⩾ 0 }$ be a standard Pisot numeration that preserves
  zeros, and let the function~$\rval$ be as defined in~\eqref{eq:def-Phi}.
  We have $\rval(ℍ_U) ⊂ 𝕃$ and~$\rval$ induces a
  function~$\tval: \uspace → (ℂ^k× ℝ^ℓ)/𝕃$.
\end{proposition}

Note that Messaoudi \cite{Messaoudi-1998} has explicitly computed an
automaton which outputs classes with more than one $U$-normalised element
for the Tribonacci $β$-numeration. More precisely, he computed an automaton
accepting pairs of $U$-normalised sequences being mapped to the same point
by~$\rval$, that is, their difference is mapped to~$0$ by~$\rval$.  Being
mapped to the same value by~$\rval$ is not the same as being in the same
class of~$\uspace$, as Proposition~\ref{pro:kernel} tells us that their
difference is mapped to the lattice~$𝕃$.

In order to prove Proposition~\ref{pro:kernel}, we need the following
results which establish a link between the function~$\rval$ and the
lattice~$𝕃$. Note that in Lemma~\ref{lem:restack_for_u_n},
Proposition~\ref{prop:restack_for_g} and Corollary~\ref{prop:rval} it is
crucial that we take standard initial conditions to define~$U$.  The
intuition behind the proof is that, given a finite sequence $g$, we
re-write it as a finite sequence $\widetilde{g}$ of length $d$, with
support in the entries with indices $0, -1, …, 1-d$, so that
$[g]_U= [\widetilde{g}]_U$. With these initial conditions, the $U$-value of
$g$ is determined by the entry $\widetilde{g}_0$.  Thus, if
$[g]_U = [g']_U$, and $u_0=1$ then we must have
$\widetilde{g}_0=[g]_U= [g']_U =\widetilde{g}'_0$, which is what leads to
$\rval(g-g') ∈ 𝕃$.

\begin{lemma}\label{lem:restack_for_u_n}
  Let $U=(u_n)_{n⩾ 0 }$ be a standard Pisot numeration, associated to
  $P(X) = X^d - a_{d-1}X^{d-1}-a_{d-2}X^{d-2}- ⋯ - a_1X -a_0$ and
  $α ∈ \{α_i: 0 ⩽ i ⩽ d-1\} $ be any one of its
  roots. For $n ⩾ d$,
  \begin{displaymath}
    α^n = ∑_{i=0}^{d-1} V_i^{(n)} α^{i}
    \qquad\text{where}\qquad
    V_{i}^{(n)} = ∑_{m=0}^i a_{i-m}u_{n-d-m} \text{ for } 0 ⩽ i ⩽ d-1.
  \end{displaymath}
  \end{lemma}
  Note that this implies
  $V_{d-1}^{(n)} = ∑_{m=0}^{d-1} a_{d-1-m}u_{n-d-m} = u_{n-d+1}$. 

\begin{proof}
  We proceed by induction on $n$.
  \paragraph{Base case} Consider $n=d$, since $P(α) = 0$, we have
  $α^d = a_{d-1} α^{d-1} + a_{d-2} α^{d-2} + ⋯ + a_{0}$ and also $u_0 = 1$
  and $u_{-1} = u_{-2} = ⋯ = 0$ by definition. Therefore, for each
  $i ∈ [0, d-1]$ we can explicitly write $V_i^{(d)} = a_i$ and so
  $∑_{i=0}^{d-1} V_i^{(n)} α^i = a_{d-1} α^{d-1} + a_{d-2} α^{d-2} + ⋯ +
  a_{0} = α ^n $ holds for $n=d$.
  \paragraph{Induction step} For some $n ⩾ d$, assume that
  $α^n = ∑_{i=0}^{d-1} V_i^{(n)} α^i$ where
  $V_{i}^{(n)} = ∑_{m =0}^{i} a_{i-m}u_{n-d-m}$ for $0 ⩽ i ⩽ d-1$.  Now for
  $α^{n+1}$, using the induction hypothesis and
  $P(α) = α^d - ∑_{i=0}^{d-1} a_iα^{i} = 0$, we have
  \begin{align*}
    α^{n+1} & = ∑_{i=0}^{d-1} V_i^{(n)} α^{i+1}  \\
           & = ∑_{i=0}^{d-1} V_i^{(n)} α^{i+1}  +
               V_{d-1}^{(n)} \left(∑_{i=0}^{d-1} a_iα^{i} - α^d\right) \\
           & = a_0 V_{d-1}^{(n)} +
               ∑_{i=1}^{d-1} (a_i V_{d-1}^{(n)}  +  V_{i-1}^{(n)}) α^{i}     
  \end{align*}
  where $V_{d-1}^{(n)} = ∑_{m=0}^{d-1} a_{d-1-m}u_{n-d-m} = u_{n+1-d}$ from
  the recurrence relation for~$U$.  But from the definitions we see that
  $a_0V_{d-1}^{(n)} = a_0u_{n+1-d} = V_0^{(n+1)}$, and for
  $1⩽ i ⩽ d-1$, and
  \begin{displaymath}
    a_i V_{d-1}^{(n)} + V_{i-1}^{(n)}  =
    a_iu_{n+1-d} + ∑_{m=0}^{i-1} a_{i-1-m}u_{n-d-m} = V_i^{(n+1)} ,
  \end{displaymath}
  and therefore $α^{n+1} = ∑_{i=0}^{d-1} V_i^{(n+1)} α^i$, from which the
  result follows.
\end{proof}
\begin{proposition}\label{prop:restack_for_g}
  Let $U=(u_n)_{n⩾ 0 }$ be a standard Pisot numeration, associated to
  $P(X) = X^d - a_{d-1}X^{d-1}-a_{d-2}X^{d-2}- ⋯ - a_1X -a_0$. For each
  finite sequence $g = g_n ⋯ g_0 ∈ ℤ^*$ and each of the roots in
  $α ∈ \{α_i: 0 ⩽ i ⩽ d-1\}$,
  \begin{displaymath}
    ∑_{k=0}^n g_kα^{k+d-1} = ∑_{k = 0}^{d-1}A_k^{(g)} α^k
  \end{displaymath}
  where each $A_k^{(g)} ∈ ℤ$ and $A_{d-1}^{(g)} = [g]_U$.
\end{proposition}
The proof of this proposition follows from
  Lemma~\ref{lem:restack_for_u_n} and linearity. 
\begin{proof}
  Using Lemma~\ref{lem:restack_for_u_n} we have
  $α^n = ∑_{i=0}^{d-1} V_i^{(n)} α^{i}$ and so
  \begin{displaymath}
    ∑_{k=0}^n   g_k α^{k+d-1} = 
    ∑_{k=0}^n \left(∑_{i=0}^{d-1} V_i^{(k+d-1)} α^i \right) g_k=
    ∑_{i=0}^{d-1} \left(∑_{k=0}^{n}g_k V_i^{(k+d-1)}\right) α^i.
  \end{displaymath}	
  But $V_{d-1}^{(ℓ)} = u_{ℓ + 1-d}$, so
  $∑_{i=0}^n g_i V_{d-1}^{(d-1+i)} = ∑_{i=0}^ng_i u_i = [g]_U$.  Thus
  \begin{align*}
    ∑_{k=0}^n g_k α^{k+d-1} & = [g]_U α^{d-1} +
           ∑_{i=0}^{d-2}\left(∑_{j=0}^{n} g_jV_i^{(j+d-1)} \right) α^{i} \\
    & = [g]_U α^{d-1} + ∑_{i=0}^{d-2}A_i^{(g)}α^i
  \end{align*}
  with $A_i ∈ ℤ$ for $i ∈ [0,d-2]$.
\end{proof}
Recalling the definition of $𝕃$, we immediately have the following corollary.
\begin{corollary} \label{prop:rval}
  Let $U=(u_n)_{n⩾ 0 }$ be a standard Pisot numeration.  If $w$ and~$w'$
  are two finite sequences over~$ℤ$ such that $[w]_U = [w']_U$, then
  $\rval(w) - \rval(w') ∈ 𝕃$.
\end{corollary}
\begin{proof}[Proof of Proposition~\ref{pro:kernel}]
  Recall $(n)_U ∈ \{0,…,\max D\}^*$ for each integer~$n ⩾ 0$, where $D$ is
  the canonical digit set.  Let $h = (h_n)_{n ⩾ 0}$ be an element of~$ℍ_U$.
  Let $c$ be an integer such that $\|h\|_∞ ⩽ c$ and let $ε > 0$.  We claim
  that there exists a point $p ∈ 𝕃$ such that $\|\rval(h)-p\|_∞ < ε$. We
  choose an integer~$m$ such that $2c\alphanorm^m < ε(1-\alphanorm)$.  By
  definition of~$ℍ_U$, there exists an integer~$n$ such that $n > m$ and
  $\uval(h_{[0:n)}) > m$.  Let $g = h_{[n:∞)}$ be the tail of~$h$ starting
  at position~$n$. Then
  \begin{align*}
    \rval(h) & = \rval(h_{[0:n)}) + \alphavec^n \rval(g) \\
             & = p + \rval(\unor(h_{[0:n)})) + \alphavec^n\rval(g)
  \end{align*} 
  for some element~$p ∈ 𝕃$ by Corollary~\ref{prop:rval} since $[\unor(h_{[0:n)})]_U=[h_{[0:n)}]_U
  $.  But since $\uval(h_{[0:n)})  > m$, we have
  \begin{displaymath}
    \|\rval(h)-p\|_∞ ⩽ c(\alphanorm^m+\alphanorm^n)/(1-\alphanorm)
  \end{displaymath}
  The choice of $m$ and~$n$ guarantees that $\|\rval(h)-p\|_∞ < ε$.  Since
  this is true for each $ε > 0$ and $𝕃$ is made of isolated points,
  $\rval(h) ∈ 𝕃$ is proved and the inclusion $\rval(ℍ_U) ⊂ 𝕃$ holds.
\end{proof}

We summarise with the following commutative diagram where $π_{ℍ_U}$ and
$π_𝕃$ are the canonical maps from~$𝔾$ to~$\uspace = 𝔾/ℍ_U$ and
from~$ ℂ^k × ℝ^ℓ $ to~$ (ℂ^k × ℝ^ℓ)/ 𝕃 ≅ 𝕋^{d-1}$ respectively.

\begin{equation}\label{diag:commutative}
  \begin{tikzcd}
    𝔾 \arrow{r}{\rval} \arrow[swap]{d}{π_{ℍ_U}}
    & ℂ^k × ℝ^ℓ \arrow{d}{π_𝕃} \\
      \uspace = 𝔾/ℍ_U \arrow{r}{\tval}&  (ℂ^k× ℝ^ℓ)/𝕃 ≅ 𝕋^{d-1}
  \end{tikzcd}
\end{equation}

With $\alphavec$ as in~\eqref{eq:norm}, define the rotation
$ρ: (ℂ^k× ℝ^ℓ)/𝕃 → (ℂ^k× ℝ^ℓ)/𝕃$ by
\begin{equation}\label{def:rho}
  ρ(\mathbf{z})≔ \mathbf{z} + \alphavec^{d-1} \bmod 𝕃.
\end{equation}

The following lemma generalises what Rauzy did in \cite{Rauzy-1982},
although we work in $ℂ$ rather than $ℝ^2$, which in our opinion is more
natural.
\begin{lemma}\label{lem:kronecker-complex}
  The map $ρ$ has dense orbits in $(ℂ^k× ℝ^ℓ)/𝕃$.
\end{lemma}
\begin{proof}
  Set $Q(X)≔ P(X)/(X-β)$; then
  $Q(X) = X^{d-1}-c_{d-2}X^{d-2} - ⋯ - c_1X - c_0$ where $c_i∈ ℤ[β]$ has
  degree at most $d-i-1$ in~$β$.  By definition we have $Q(α_i) = 0$ for
  $1 ⩽ i ⩽  k+ℓ$ so
  \begin{displaymath}
    \alphavec^{d-1} = c_0\onevec + c_1\alphavec + ⋯ +  c_{d-2}\alphavec^{d-2}.
  \end{displaymath}
  Then
  \begin{displaymath}
    n\alphavec^{d-1}  ≡ (nc_0 \bmod 1)\onevec +(nc_1 \bmod 1) \alphavec + ⋯
    +(nc_{d-2} \bmod 1)\alphavec^{d-2} \mod 𝕃
  \end{displaymath}
  Thus the orbit of any $\mathbf{z} ∈ (ℂ^k× ℝ^ℓ)/𝕃 $ under
  $+\alphavec^{d-1}$ is dense in $(ℂ^k× ℝ^ℓ)/𝕃$ if and only if the orbit of
  the map
  \begin{displaymath}
    (x_0, …, x_{d-2} ) ∈  ℝ^{d-1}/ℤ^{d-1}
    ↦
    ((x_0+ c_0) \bmod 1 , …, (x_{d-2}+ c_{d-2}) \bmod 1 )
  \end{displaymath}
  is dense in the torus $ℝ^{d-1}/ℤ^{d-1}$. By Kronecker's theorem, this is
  true if and only if the $d$ numbers $1, c_0, …, c_{d-2}$
  are linearly independent over~$ℤ$.  Since each $c_i ∈ ℤ[β]$ has degree at
  most $d-1$ and $P$ is by definition irreducible, the result follows.
\end{proof}

\begin{corollary} \label{cor:dense}
  Let $U=(u_n)_{n⩾ 0 }$ be a standard Pisot numeration that preserves
  zeros, and let $\tval$ be as defined in~\eqref{diag:commutative}.  The
  set $\tval(π_{ℍ_U}(ℕ_U)) = \{\tval(π_{ℍ_U}((n))_U) : n ⩾ 0\}$ is dense in
  $(ℂ^k× ℝ^ℓ)/𝕃$.
\end{corollary}

\subsection{A topology on $𝔾$ }\label{section:topology}

We now come to the continuity of~$\tval$, which amongst other things will
let us conclude that the function~$\tval$ is onto.  We first endow the
group~$𝔾$ with an appropriate metric topology.  For $g = (g_n)_{n ⩾ 0}∈ 𝔾$,
set
\begin{displaymath}
  \|g\|_U = ∑_{n ⩾ 0} |g_n|\alphanorm^n .
\end{displaymath}
We endow the group $𝔾$ with the topology defined by the distance~$d$
given by
\begin{displaymath}
  d(g, g') = \|g-g'\|_U.
\end{displaymath}
It is clear that this topology makes the function~$\tval$ continuous
because
\begin{displaymath}
  \|\rval(g)-\rval(g')\|_∞ ⩽ \|g-g'\|_U = d(g, g')
\end{displaymath}
holds for all $g,g' ∈ 𝔾$.  Note that this topology makes addition in $𝔾$
continuous as
\begin{displaymath}
  \|(f+g)-(f'+g')\|_U ⩽ \|f-f'\|_ U+ \|g-g'\|_U
\end{displaymath}
for  all $f,f',g,g' ∈ 𝔾$.

Note that the distance~$d$ that we have just defined is not equivalent to
the distance~$d'$ given by $d'(g, g') = 2^{-\min\{n : g_n ≠ g'_n\}}$.  To
see this, we will define a sequence $(g^{(n)})$ of elements of~$𝔾$ such
that $d(g^{(n)}, \period{0}1) → 0$ while $d'(g^{(n)}, \period{0}) → 0$
as $n → ∞$.  Let $U = (u_n)_{n ⩾ 0}$ be the Fibonacci numeration with
$u_0 = u_1 = 1$ and $u_{n+2} = u_{n+1} + u_n$ for $n ⩾ 0$.  By Binet's
formula, each number~$u_n$ is equal to
$(\gmean^{n+1}-\hmean^{n+1})/{\sqrt{5}}$ where $\gmean$ is the golden mean
and $\hmean$ its algebraic conjugate.  For each integer~$n ⩾ 0$, let
$g^{(n)}$ be defined as follows.
\begin{displaymath}
  g_k^{(n)} =
  \begin{cases}
    \gmean^{2n} + \hmean^{2n} & \text{if $k = 2n$}, \\
    0                        & \text{otherwise}.
  \end{cases}
\end{displaymath}

Note first that $\gmean^{2n} + \hmean^{2n}$ is indeed an integer and
  that each element~$g^{(n)}$ has only one non-zero entry.  The
first elements are $g^{(0)} = \period{0}2$, $g^{(1)} = \period{0}300$ and
$g^{(2)} = \period{0}70000$.  It can be verified
that$[g^{(n)}]_U = u_{4n} + 1$ for each $n ⩾ 0$.  It follows that
$d(g^{(n)},\period{0}1)→ 0$ but $d'(g^{(n)}, \period{0})→ 0$ as $n → ∞$.
 
However, for any finite subset~$B$ of~$ℤ$, the restrictions of $d$ and~$d'$
to $B^ℕ$ are equivalent.  We endow $\uspace = 𝔾/ℍ_U$ with the quotient
topology.  Note that the map $+1 : \uspace → \uspace$ defined by
\begin{displaymath}
  [z]+1 ≔ [z]+[\period{0}1] = [z+\period{0}1]  
\end{displaymath}
is well defined and continuous. Note that $\uspace$ is  compact, by Proposition~\ref{pro:trace}.

\begin{proposition}\label{prop:surjective}
  The function~$\tval: \uspace → (ℂ^k× ℝ^ℓ)/𝕃$ is continuous. \end{proposition}
\begin{proof}
  The function~$\rval:𝔾 → ℂ^k× ℝ^ℓ$ is continuous by definition
  of~$\|\,\,\|_U$.  The canonical map~$π_𝕃:ℂ^k× ℝ^ℓ → (ℂ^k× ℝ^ℓ)/𝕃$ is also
  continuous, thus so is their composition $π_𝕃 ∘ \rval$.  Let $O$ be an
  open set of $(ℂ^k× ℝ^ℓ)/𝕃$.  By the commutative
  diagram~\eqref{diag:commutative},
  $(\tval ∘ π_{ℍ_U})^{-1}(O)= (π_𝕃 ∘ \rval)^{-1}(O)$ and it is thus open as
  $π_𝕃 ∘ \rval$ is continuous.  By definition of the quotient topology,
  $\tval^{-1}(O)$ is open and $\tval : \uspace → (ℂ^k× ℝ^ℓ)/𝕃$ is thus a
  continuous function.
\end{proof}

\subsection{ What unimodularity yields}\label{sec:iso}

Notice that thus far we have not assumed that $β$ is unimodular.  Under
certain conditions, the function~$\tval$ defined
in~\eqref{diag:commutative} is one-to-one.  If the following two
assumptions are satisfied by the numeration~$U$, then we can show
that~$\tval$ is injective.

\begin{itemize}
\item[\textbf{Hyp.~1}] The unique $U$-normalised  sequence mapped
  to~$\mathbf{0}$ by~$\rval$ is the zero sequence $⋯000 = \period{0}$. In
  other words,
  \begin{displaymath}
    \rval^{-1}(\mathbf{0}) ∩ \overline{ℕ}_U = \{\period{0}\}.
  \end{displaymath}

\item[\textbf{Hyp.~2}] The distance between~$𝕃 ∖ \{\mathbf{0}\} $ and
  the images under~$\rval$ of the $U$-normalised  sequences is positive. In
  other words,
  \begin{displaymath}
    d(\rval(\overline{ℕ}_U), 𝕃 ∖ \{\mathbf{0}\}) =
    \inf_{\begin{smallmatrix} n ∈ ℕ \\ \mathbf{p} ∈ 𝕃∖ \{\mathbf{0}\} \end{smallmatrix}} 
    \|\rval((n)_U) - \mathbf{p}\|_∞ > 0
  \end{displaymath}
\end{itemize}

\begin{theorem}\label{thm:Phi-isomorphism}
  Let $U=(u_n)_{n⩾ 0 }$ be a standard Pisot numeration that preserves
  zeros, and let $\rval$ and $\tval$ be as defined
  in~\eqref{diag:commutative}.  If $U$ satisfies Hypotheses~1 and~2, then
  $\tval: \uspace → (ℂ^k× ℝ^ℓ)/𝕃 $ is a continuous group isomorphism. In
  particular, this gives a topological conjugacy between $(\uspace,+1)$ and
  $((ℂ^k× ℝ^ℓ)/𝕃 ,ρ)$.
\end{theorem}
\begin{proof}
  By Proposition~\ref{prop:surjective}, $\tval$ is a continuous group
  morphism.  By Proposition~\ref{pro:trace}, each class of
  $\uspace = 𝔾/ℍ_U$ contains an element in~$\overline{ℕ}_U$.  Therefore, we
  can identify elements of~$\uspace$ with elements of~$\overline{ℕ}_U$.  By
  Corollary~\ref{cor:dense}, $\tval(π_{ℍ_U}(ℕ_U))$ is dense in
  $(ℂ^k× ℝ^ℓ)/𝕃$.  Therefore, $\tval$ is onto.  It remains to show that
  $\tval$ is one-to-one.  Let $g ∈ \overline{ℕ}_U$ be such that
  $\rval(g) ∈ 𝕃$.  By Hyp.~2, $\rval(g)$ must be equal to~$\mathbf{0}$.  By
  Hyp.~1, $g$ must be the sequence $\period{0}$.  This shows that the
  kernel of~$\rval$ is $ℍ_U$ and that $\tval$ is one-to-one.  Since
  $\uspace$ and $(ℂ^k× ℝ^ℓ)/𝕃$ are compact spaces, $\tval^{-1}$ is also
  continuous.
\end{proof}

Next we will show that Hypotheses 1 and 2 are satisfied if $β$ is
unimodular and if the standard numeration~$U$ associated to~$β$ preserves
zeros.  Recall the $β$-numeration defined in
Section~\ref{sec:beta-numeration}.  Define
\begin{displaymath}
  \fin(β)≔  \{ x⩾ 0: x \text{ has a finite $β$-expansion} \}.
\end{displaymath}
The Pisot number~$β$ \emph{satisfies the Condition~F} if $ℤ[β] ⊂ \fin(β)$
\cite{Frougny-Solomyak-1992}.  The Condition~F is sometimes written
$ℤ[1/β] ⊂ \fin(β)$.  The following elementary lemma shows that the two
conditions are indeed equivalent. Note that if $α ∈ \fin(β)$, then both
$αβ$ and $α/β$ also belong to~$\fin(β)$.
\begin{lemma}\label{lem:def-F}
  \begin{displaymath}
    ℤ[β] ⊂ \fin(β) \iff ℤ[1/β] ⊂ \fin(β)
  \end{displaymath}
\end{lemma}
\begin{proof}
  We prove one direction as the reverse implication is proved the same way.
  Suppose that $ℤ[β] ⊂ \fin(β)$ and let $α ∈ ℤ[1/β]$, so that
  $α = ∑_{i = 0}^n a_iβ^{-i}$ where each coefficient $a_i ∈ ℤ$.  Thus
  $αβ^n = ∑_{i = 0}^n a_iβ^{n-i} ∈  ℤ[β]$. Therefore $αβ^n$ has a finite
  $β$-expansion and so has $α$.
\end{proof}

\begin{theorem} \label{thm:hyp12}
  Let $U$ be the standard numeration associated to a unimodular Pisot
  number~$β$.  If $β$ satisfies Condition~F, then $U$ satisfies Hyp.~1
  and~2.
\end{theorem}

We prove in Proposition~\ref{prop:zero-preserving-then-F} that if $U$
preserves zeros, then $\beta$ satisfies Condition~F.  We deduce
\begin{corollary} \label{cor:hyp12}
  Let $U$ be the standard numeration associated to a unimodular Pisot
  number~$β$.  If $U$ preserves zeros, then $U$ satisfies Hypotheses~1
  and~2.
\end{corollary}

\subsubsection{Notation}
\label{sec:notation}

Let $U = (u_n)_{n ⩾ 0}$ be a Pisot numeration of degree~$d$, with standard
initial conditions, and which preserves zeros. As before, let
$α_0 = β, α_1, …, α_{d-1}$ denote the conjugates of~$β$, and let $m$ be the
length of (necessarily finite) $d_β(1)$.  Let
\begin{displaymath}
  \gtwo ≔ \{ g =(g_n)_{n ∈ ℤ} = ⋯  g_1g_0 ⋅ g_{-1}g_{-2}⋯  ∈ ℤ^ℤ
                 \text{ with } \|g\|_∞ < ∞  \}
\end{displaymath}
be the inverse limit of~$𝔾$.  Here, as elsewhere in this article, positive
indices run to the left, and negative indices run to the right.  For
$g∈ \gtwo$, as for one-sided sequences, define
\begin{displaymath}
  \ord(g) ≔ \inf \{ n : g_n ≠ 0\}
  \quad\text{and}\quad
  \deg(g) ≔ \sup \{ n : g_n ≠ 0\}.
\end{displaymath}
As usual the convention is that $\ord(\period{0})= ∞$ and
$\deg(\period{0})= -∞$.  For example,
$\ord( \period{0}321⋅ 1234 \period{0})= -4$ and
$\deg( \period{0}321⋅ 1234 \period{0})= 2$.  Let
\begin{displaymath}
  \gfin ≔ \{ g ∈ \gtwo:  \ord(g) > -∞  \text{ and }  \deg(g) < ∞\}
\end{displaymath}
be the space of two-sided finite sequences $g = (g_n)_{n = -∞}^∞$.

For a complex number~$α$, and for $g ∈ \gfin$, we set
$[g]_α≔ ∑_{n = -∞}^∞ g_nα^n$; this is well-defined as the latter is a
finite sum.  We mainly use the notation~$[g]_α$ when $α$ is one of~$β$'s
conjugates.  If the finite $g$ satisfies $\ord(g) ⩾ 0$, then we can
identify it as an element of~$𝔾$, and $[g]_U=∑_{n = 0}^∞ g_nu_n$.  Since
$β$ is unimodular, each integer~$u_n$ for a negative index~$n$ is
well-defined, for the recurrence can hold for all $n ∈ ℤ$ except for the
specified initial conditions ($u_0, u_{-1}, …, u_{-d+1}$), but,
nevertheless, we only consider $U$-expansions for~$g$ with $\ord(g) ⩾
0$. If $\ord(g) ⩾ 0$, $g = (g_n)_{n = 0}^∞$ is \emph{$U$-normalised} if
$∑_{n = 0}^k g_nu_n < u_{k+1}$ for each integer $k ⩾ 0$.
  
The sequence $g = (g_n)_{n = -∞}^∞$ is \emph{$β$-normalised} if
$∑_{n = -∞}^k g_nβ^n < β^{k+1}$ for each integer $k ∈ ℤ$. Note here that
the sum is convergent because $\|g\|_∞ < ∞$.

The \emph{$β$-expansion} $(x)_β$ of~$x ⩾ 0$ is the unique $β$-normalised
sequence $g = (g_n)_{n = -∞}^∞$ such that $x = [g]_β$.  This does not have
to be a finite expansion, but it is the case if $(x)_β= g ∈ \gfin$ and $β$
satisfies the Condition~F.  Both the $β$-expansion of $x⩾ 0$ and the
$U$-expansion of $k ∈ ℕ$ can be obtained using the greedy algorithm.

The \emph{$β$-normalisation} $\nor_β(g)$ of a sequence
$g = (g_n)_{n = -∞}^∞$ such that $\deg(g) < ∞$ is the sequence
$(|[g]_β|)_β$. The assumption that~$β$ satisfies Condition~F, equivalently,
that $U$ preserves zeros, guarantees that the $β$-normalisation of a finite
sequence is still a finite sequence.

The shift map~$σ: \gtwo → \gtwo$ is defined by
$σ((g_n)_{n = -∞}^∞) ≔ (g_{n-1})_{n = -∞}^∞$, i.e.,
$σ(⋯ g_1 g_0⋅ g_{-1} g_{-2}⋯ ) ≔ ⋯ g_1 g_0 g_{-1}⋅ g_{-2}g_{-3}⋯ $.  Note
that this is not the classical definition of the shift map, but rather its
inverse.  If $g ∈ \gfin$, then
\begin{displaymath}
  \ord(σ(g)) = \ord(g) + 1, \, \deg(σ(g)) = \deg(g) + 1
  \quad\text{ and }\quad
  [σ(g)]_α = α[g]_α
\end{displaymath}
for each~$α ∈ \{α_0 , α_1, …, α_{d-1} \}$.

\begin{lemma} \label{lem:unormbnorm-take2} 
  Let $U = (u_n)_{n ⩾ 0}$ be a Pisot numeration associated to~$β$.  There
  exists an integer $N = N(U)$ such that if the  sequence
  $g = (g_n)_{n = 0}^∞$ satisfies $\ord(g) ⩾ N$ then $g$ is $U$-normalised
  if and only if $g$ is $β$-normalised. 
\end{lemma}

Recall the definitions of $d_β(1)$ and $d_β^*(1)$ from
Section~\ref{sec:beta-numeration}.  The sequence
$d_β^*(1) = t_1^*t_2^*t_3^* ⋯$ satisfies
\begin{enumerate}
\item $∑_{i ⩾ 1}t_i^* β^{-i} = 1$, and 
\item there exists $η > 0$ such that $∑_{i ⩾ 1}t_{i+n}^*β^{-i} > η$ for each
  $n ⩾ 1$.
\end{enumerate}
The last property follows from the fact that the sequence $d_β(1)$ and
$d_β^*(1)$ are both ultimately periodic and that the sum
$∑_{i ⩾ 1}t_{i+n}^*β^{-i}$ takes finitely many non-zero values as $n$
varies.

\begin{proof} 
  Let $λ>0$ and $K>0$ be the real numbers guaranteed by
  Lemma~\ref{lem:boundun}. Recall that $D$ denotes the canonical digit
  set. Choose $N$ such that $K (\max D)\alphanorm^{N}/(1-\alphanorm) < 1$
  and $ληβ^N>1$.
  
  We first prove that if $g$ is $U$-normalised, then $g$ is $β$-normalised.
  For each $n$ such that $n ⩾ N$, using Lemma~\ref{lem:boundun} for
  the first and third inequalities, and the fact that $g$ is $U$-normalised
  for the second, we have
  \begin{align*}
    λ∑_{i = N}^ng_iβ^i
      & ⩽ ∑_{i = N}^ng_iu_i + K(\max D) ∑_{i = N}^n\alphanorm^i \\
      & ⩽ u_{n+1} -1 + K(\max D)∑_{i = N}^n\alphanorm^i \\
      & ⩽ λβ^{n+1} -1 + K(\max D)∑_{i = N}^{n+1}\alphanorm^i \\
      & ⩽ λβ^{n+1} -1 + K(\max D)\alphanorm^{N}/(1-\alphanorm).
  \end{align*}
  The choice of $N$ now gives that $g$ is $β$-normalised.

  Next suppose that $g$ is $β$-normalised. To see that $g$ is
  $U$-normalised, we first write
  \begin{align*}
    ∑_{i = N}^ng_iu_i
      & = ∑_{i = N}^ng_i(u_i - λβ^i) + λ ∑_{i = N}^ng_iβ^i.
  \end{align*}
  We claim that $∑_{i = N}^ng_iβ^i ⩽ β^{n+1}-ηβ^N$ for each $n ⩾ N$. If
  this is true, then again applying Lemma~\ref{lem:boundun}, we have
  \begin{align*}  ∑_{i = N}^ng_iu_i
       &⩽ K\|g\|_∞ ∑_{i = N}^n\alphanorm^i + λ(β^{n+1}-ηβ^N) \\
      & ⩽ K\|g\|_∞ ∑_{i = N}^n\alphanorm^i + u_{n+1}
          + K\alphanorm^{n+1} - ληβ^N\\
      & ⩽ u_{n+1} + \frac{K\|g\|_∞\alphanorm^N}{1-\alphanorm} - ληβ^N
  \end{align*}
  By choice of $N$, we have $K\|g\|_∞\alphanorm^N < ληβ^N(1-\alphanorm)$
  and $g$ is $U$-normalised.
  
  To prove the claim, we have $∑_{i = N}^ng_iβ^i ⩽ β^{n+1}-ηβ^N$ if and
  only if $∑_{i = 0}^{n-N} g_{i+N}β^i ⩽ β^{n-N+1}-η$; rearranging, we will
  show that $η ⩽ β^{n-N+1}-∑_{i = 0}^{n-N} g_{i+N}β^i$.  By definition
  $∑_{i ⩾ 1}t_{i+k}^*β^{-i} > η$ for each $k⩾ 1$. This, with the fact that
  $∑_{i ⩾ 1}t_i^* β^{-i} = 1$, implies that
  $η<β^{k}- β^{k-1}t_1-\dots - βt_{k-1}-t_{k}$ for each $k⩾ 1$. Setting
  $k=n+1-N$, we have
  \begin{align*}
     η & < β^{n-N+1}- β^{n-N}t_1- ⋯ - βt_{n-N}-t_{n+1-N}\\
       &⩽ β^{n-N+1}- β^{n-N}g_n-\dots - βg_{N+1} -g_N,
  \end{align*}
  where the last inequality follows because
  $g_n⋯ g_N \ltlex t_1 ⋯ t_{n+1-N}$ by Theorem~\ref{thm:Pisot-properties}.
\end{proof}

\begin{lemma} \label{lem:boundvalnor}
  If $g ∈  \gfin$, $\ord(g) ⩾ 0$ and $g$ is $U$-normalised, then
  $\ord(\nor_β(g)) ⩽ \max(\ord(g), N)$, where $N$ is given by
  Lemma~\ref{lem:unormbnorm-take2}.
\end{lemma}
\begin{proof}
  We will show that either $\ord(\nor_β(g)) = \ord(g)$, or
  $\ord(\nor_β(g)) ⩽ N$. This gives the claim.

  If $\ord(g) ⩾ N$, then $\nor_β(g) = g$ by Lemma~\ref{lem:unormbnorm-take2},
  and thus $\ord(\nor_β(g))= \ord(g)$.
  
  Suppose that $\ord(g) < N$ but that $\ord(\nor_β(g)) ⩾ N$. Then, again by
  Lemma~\ref{lem:unormbnorm-take2}, $\nor_β(g)$ is $U$-normalised.  Since
  $[g]_β = [\nor_β(g)]_\beta$, we also have $[g]_U = [\nor_β(g)]_U$.  This
  is a contradiction because, $g$ and $\nor_β(g)$ are both $U$-normalised.
\end{proof}

The following lemma is the key lemma. It tells us
  that when we add and normalise, the first non-zero entry cannot move too
  far.
\begin{lemma} \label{lem:boundedval}
    Let $g\in \gfin$ be a $U$-normalised sequence with $\ord(g) ⩾ 0$ and let
  $g' = (g'_n)_{n = 0}^{d-2}$ have support in $\{0, …, d-2\}$. Let $N$ be
  given by Lemma~\ref{lem:unormbnorm-take2}.  If at least one entry of~$g'$
  is non-zero, then $\ord(\nor_β(σ^{d-1}(g) + g')) < N+d$.
\end{lemma}
Note that if $g = ⋯ g_2g_1g_0⋅$ and $g' = g'_{d-2} ⋯ g'_0⋅$, then
$σ^{d-1}(g) + g' = ⋯ g_2g_1g_0g'_{d-2} ⋯ g'_0⋅$, obtained by concatenating
$g$ and~$g'$.  Note that the statement fails if all entries of~$g'$ are
equal to zero.  For, in this case, if $\ord(g) > N+d$ then $g$ is also
$β$-normalised, $\nor_β(σ^{d-1}(g) + g') = σ^{d-1}(g)$, and
$\ord(σ^{d-1}(g)) = \ord(g) + d -1 ⩾ N+2d-1$.
\begin{proof}
  Suppose by contradiction that $\ord(\nor_β(σ^{d-1}(g) + g')) ⩾ N+d$.
  Note that $\deg(\nor_β(σ^{d-1}(g) + g'))<\infty$, so
  $\nor_β(σ^{d-1}(g) + g')\in \gfin$.  Let
  $g''=\sigma^{1-d}(\nor_β(σ^{d-1}(g) + g'))$.  Then
  $\ord(g'') ⩾ N+d+(1-d) ⩾ N$.  The sequence~$g''$ is $β$-normalised as
  $\nor_β(σ^{d-1}(g) + g')$ is and shifting preserves
  $β$-normalisation. Since $\ord(g'') ⩾ N$, the sequence~$g''$ is also
  $U$-normalised by Lemma~\ref{lem:unormbnorm-take2}.  We apply
 Proposition~\ref{prop:restack_for_g}. Combining
  \begin{alignat*}{2}
    [σ^{d-1}(g)]_{α_j}  & = ∑_{i = 0}^{d-1} A_i^{(g)}α_j^i
                      & \quad &\text{for } 0 ⩽ j ⩽ d-1 \\
    [σ^{d-1}(g'')]_{α_j}  & = ∑_{i = 0}^{d-1} A_i^{(g'')}α_j^i
                        & \quad &\text{for } 0 ⩽ j ⩽ d-1 \\
    [σ^{d-1}(g'')]_{α_j} - [σ^{d-1}(g)]_{α_j} = [g']_{α_j}
        & = ∑_{i = 0}^{d-2} g'_iα_j^i & \quad &\text{for } 0 ⩽ j ⩽ d-1 
  \end{alignat*}
  we conclude that $[g]_U = A_{d-1}^{(g)} = A_{d-1}^{(g'')} = [g'']_U$.
  Since both $g$ and $g''$ are $U$-normalised, it follows that $g = g''$.
  This is a contradiction because $[g'']_β -[g]_β = [g']_ββ^{1-d} ≠ 0$.
\end{proof}

The following lemma is cited in \cite{Akiyama-1997} without proof.
We provide the proof for convenience of the reader.
\begin{lemma} \label{lem:boundedset} 
  Let $β$ be a unimodular Pisot number of degree~$d$ whose conjugates are
  $α_0 = β,α_1,…,α_{d-1}$.  For each positive real number~$M ⩾ 0$, the set
  \begin{displaymath}
    S = \left\{ g  ∈  \gfin : \text{$g$ is $\beta$-normalised} \text{
            and } |[g]_{α_j}| ⩽ M \text{ for } 0 ⩽ j ⩽ d-1\right\} 
  \end{displaymath}
  is finite.
\end{lemma}
\begin{proof}
  Let $g = (g_n)_{n = -∞}^∞$ be a finite sequence.  The fact that $β$ is
  unimodular is used to express $β^n$ for negative~$n$ as a linear integer
  combination of $1,β,…,β^{d-1}$ with integer coefficients. Thus, there are
  $d$ integers $(n_i)_{i=0}^{d-1}$ such that $[g]_β=∑_{i=0}^{d-1} n_iβ^i$.
  It follows that each conjugate~$α_j$ of~$β$ for $1 ⩽ j ⩽ d-1$ satisfies
  $[g]_{α_j} = ∑_{i=0}^{d-1} n_iα_j^i$.  These $d$ equalities can be
  rewritten as the following matrix equality:
  \begin{displaymath}
    ([g]_{α_0},…,[g]_{α_{d-1}}) = (n_0,…,n_{d-1})
    \begin{pmatrix}
      1 & 1 & ⋯ & 1 \\
      α_0 & α_1 & ⋯ & α_{d-1} \\
      α_0^2 & α_1^2 & ⋯ & α_{d-1}^2 \\
      \vdots  & \vdots &   & \vdots  \\
      α_0^{d-1} & α_1^{d-1} & ⋯ & α_{d-1}^{d-1}
    \end{pmatrix}.
  \end{displaymath}
  Since the Vandermonde matrix $M = (α_j^i)_{0 ⩽ i,j ⩽ d-1}$ is invertible
  and, if the entries of the vectors $([g]_{α_0},…,[g]_{α_{d-1}})$ are
  bounded by $M$ as $g$ varies, the vectors of integers $(n_0,…,n_{d-1})$
  also have bounded entries and therefore there are finitely many possible
  vectors.
\end{proof}

\begin{lemma} \label{lem:boundedphi}
  There is a real number $K > 0$ such that for any nonzero $β$-normalised
  sequence $g\in\gfin$ with $\ord(g)⩾ 0$,
  \begin{displaymath}
    \max \{ |[g]_{α_i}| : 1 ⩽ i ⩽ d-1\} ⩾ K
        \min \{ |α_i| : 1 ⩽ i ⩽ d-1 \}^{\ord(g)}.
  \end{displaymath}
\end{lemma}
The proof of this lemma is borrowed from~\cite{Akiyama-1997}. 
\begin{proof}
  Since $β > 1$ there is a constant~$M_0$ such that $\deg(g) ⩽ 0$ implies
  $[g]_β < M_0$ for each finite normalised sequence $g$.  Since $|α_i| < 1$
  for each $1 ⩽ i ⩽ d-1$, there is a constant~$M_1$ such that
  $\ord(g) ⩾ 0$ implies $|[g]_{α_i}| < M_1$ for each finite sequence
  $g = (g_n)_{n = -∞}^∞$ and each~$i$.  Let $M ≔ max(M_0, M_1)$.

  A finite sequence $g$ can be decomposed as a sum of its left and right
  infinite parts, $g = g' + g''$ where $\ord(g') ⩾ 0$ and $\deg(g'') < 0$.
  By definition of~$M$, $[g'']_β < M$ and $|[g']_{α_i}| < M$ for
  $1 ⩽ i ⩽ d-1$.  By Lemma~\ref{lem:boundedset}, the set
  \begin{displaymath}
    S ≔ \left\{ g ∈ \gfin: \text{ $g$ is $β$-normalised},
                 |[g]_{α_i}| ⩽ 2M \text{ for each } i ∈ \{0,…,d-1\}\right\}
  \end{displaymath}
  is finite, since each finite $β$-normalised sequence has distinct
    $[g]_β$.  Let $L ≔ \min( \{ \ord(g) : g ∈ S\} \cup\{ -1\})$.  We claim
  that if $\ord(g) < L$, then there exists an index~$i$ such that
  $|[g]_{α_i}| > M$.  To see this, suppose that $\ord(g) <L$.  Since
  $\ord(g)<L<0$, we have $ \ord(g'')=\ord(g)<L$ and so~$g''\notin S$, by
  definition of $L$.  Since $[g'']_β < M$, there exists an integer
  $1 ⩽ i ⩽ d-1$ such that $|[g'']_{α_i}| > 2M$.  Since $|[g']_{α_i}| < M$,
  it follows that $|[g]_{α_i}| ⩾ |[g'']_{α_i}| - |[g']_{α_i}| > M$.  This
  proves our claim.
  
  Let $g$ be a nonzero finite $β$-normalised sequence and let
  $g^* ≔ σ^{L-\ord(g)-1}(g)$. By definition $\ord(g^*) = L-1$ and thus
  there exists an index~$i$ such that $|[g^*]_{α_i}| > M$ and
  $|[g]_{α_i}| = |α_i|^{1+\ord(g)-L} |[g^*]_{α_i}| > |α_i|^{1+\ord(g)-L}M$.
  Setting $K=|α_i|^{1-L}M$, we have $K>0$ and
  $|[g]_{α_i}| >K |α_i|^{\ord(g)}$.  The result follows.
\end{proof}

\begin{proof}[Proof of Theorem~\ref{thm:hyp12}]
  We prove Hyp.~1. First let $g$ be a finite $U$-normalised sequence with
  $\ord(g) ⩾ 0$.  By Lemma~\ref{lem:boundvalnor},
  $\ord(\nor_\beta(g)) ⩽ \ord(g) + N$. Let
  $λ = \min \{ |α_i| : 1 ⩽ i ⩽ d-1 \}$. Since $g$ is not identically
  zero, there exists, by Lemma~\ref{lem:boundedphi}, an index~$i$ such that
  \begin{equation} \label{eq:finite}
    |[g]_{α_i}| = |[\nor_β(g)]_{α_i}| > Kλ^{\ord(g)+N}.
  \end{equation}  
  Note that here we are using the Condition~F twice: when asserting the
  equality, for otherwise $|[\nor_β(g)]_{α_i}|$ may not converge, and also
  because then we cannot apply Lemma~\ref{lem:boundedphi}.  Therefore
  $\rval(g) ≠ 0$. Now if $g ∈ \overline{ℕ}_U$, \eqref{eq:finite} is true
  for any truncation of $g^{(n)} = g_n \cdots g_0$ of $g$, with
  $\ord(g^{(n)})=\ord(g) $ and
  $\ord(\nor_\beta(g^{(n)})) ⩽ \ord(g) + N$ for any $n$. The result
  follows since there is an $i$ such that
  $|[g]_{α_i}|= \lim_{n→∞} |[g^{(n)}]_{α_i}|⩾  Kλ^{\ord(g)+N} $.

  To prove Hyp.~2,  we must show that
  \begin{displaymath}
    \inf \left\{\|\rval((n)_U) - \mathbf{p}\|_∞: n ∈ ℕ, \mathbf{p} ∈ 𝕃∖ \{\mathbf{0}\}  \right\} > 0.
  \end{displaymath}
  Let $\mathbf{p} = ∑_{i=0}^{d-2}g'_iα^i$ with
  $\mathbf{p} ≠ \mathbf{0} ∈  \mathbb{L}\backslash \{\mathbf{0}\}$.  Let
  $g'=(-g'_i)_{i=0}^{d-2}$; by assumption it has a nonzero entry.  Let $g$
  be a finite $U$-normalised sequence with $\ord(g) ⩾ 0$.  By
  Lemma~\ref{lem:boundedval}, $\ord(\nor_β(σ^{d-1}(g) + g')) < N+d$.  By
  Lemma~\ref{lem:boundedphi}, for some $i$ we have
  $|[\nor_β(σ^{d-1}(g) + g')]_{α_i}| ⩾ Kλ^{N+d}$.  This
  implies that $\|\rval (g) - \mathbf{p}\|_∞ ⩾ Kλ^{N+d}$.
  Here also we need Condition~F. This is true for any finite $U$-normalised
  $g$.
\end{proof}

\section{ Condition F}\label{sec:condition-F}

In this section we show that our condition of preservation of zeros is
equivalent to the now classical Condition~F. Recall that we used this to prove
Corollary~\ref{cor:hyp12}.

\subsection{From preservation of zeros to Condition F}

The set of $U$-expansions of natural numbers for a Pisot numeration is
clearly related to the set of finite $β$-expansions for corresponding
$β$-numeration; for example, one can mechanically normalise the same way
for both the $U$- and $β$-numeration. However one can run into issues close
to the radix point. For example, with our choice of initial conditions for
the Zeckendorf numeration, all normalised expansions start with~$0$, which
is not the case for the $β$-numerations for $β = \gmean$.  Nevertheless,
provided we are sufficiently far from the radix point, these sets of
expansions are the same.

\begin{lemma} \label{lem:graph-cycle}
  Let $U = (u_n)_{n ⩾ 0}$ be a Pisot numeration, and let $g$ be a finite
  sequence. There exists an integer $N=N(U, \|g\|_∞)$ such that if there
  exists $n ⩾ N$ such that $[g0^n]_U=0$, then there exist infinitely many
  $n ⩾ N$ with $[g0^n]_U=0$.
\end{lemma}
\begin{proof}
  Let $B = \{-\|g\|_∞, …, \|g\|_∞\}$.  By Theorem~\ref{thm:regularity-0},
  there exists an automaton over the alphabet~$B$ accepting the set
  $\{ w ∈ B^* : [w]_U = 0\}$.  Let $N = N(U, \|g\|_∞)$ be the number of
  states of this automaton.  If $[g0^n]_U = 0$, then the finite sequence
  $g0^n$ is accepted by the automaton.  Now if furthermore $n ⩾ N$, then
  the pumping lemma gives the result.
\end{proof}

For a finite sequence $g = g_k ⋯ g_0$, let $P_g$ be the polynomial
$P_g(X) = ∑_{i = 0}^kg_iX^i$.  Note that for finite sequences $g$ and~$g'$,
$P_{g+g'} = P_{g} + P_{g'}$.
\begin{lemma} \label{lem:dividesPw}
  Let $U = (u_n)_{n ⩾ 0}$ be a Pisot numeration with minimal
  polynomial~$P$, and let $g$ be a finite sequence.   Then there exists an
  integer $N = N(U, \|g\|_∞)$ such that if $\ord(g) ⩾ N$ and $[g]_U = 0$,
  then $P$ divides $P_g$.
\end{lemma}
\begin{proof}
  Let $g$ be the finite sequence $g_k ⋯ g_0$ and recall that
  $α_0 = β, α_1, …, α_{d-1} $ are the roots of~$P$.  There exist
  coefficients $(λ_i)_{0⩽ i ⩽ d-1}$ such that $u_n = ∑_{i=0}^{d-1}λ_iα_i^n$
  for each $n ⩾ 0$.  Recall the definition of $\alphanorm$
  from~\eqref{eq:norm}.  We have
  \begin{align*}
    \left|[g0^n]_U -λ_0β^nP_g(β)\right|
      & = ∑_{j=0}^k g_j(u_{n+j} -λ_0β^{n+j}) \\
      & = ∑_{j=0}^k g_j \left(∑_{i=0}^{d-1} λ_iα_i^{n+j} -λ_0β^{n+j}\right) \\
      & = ∑_{j=0}^k g_j∑_{i=1}^{d-1} λ_iα_i^{n+j} \\
      & ⩽ ∑_{j=0}^k |g_j|∑_{i=1}^{d-1} |λ_i|\alphanorm^{n+j} \\
      & ⩽ \alphanorm^n\|g\|_∞∑_{j=0}^∞∑_{i=1}^{d-1} |λ_i|\alphanorm^j \\
      & ⩽ \frac{K\alphanorm^n}{1-\alphanorm} 
  \end{align*}
  where $K = (d-1)\|g\|_∞\max_{1 ⩽ i ⩽ d-1} |λ_i|$.  By
  Lemma~\ref{lem:graph-cycle}, $[g0^n]_U = 0$ for some sequence of
  integers~$n$. Since $K\alphanorm^n/(1-\alphanorm)$ converges to zero when
  $n → ∞$, and $β>1$, $P_g(β)=0$.  Since $P$ is the minimal polynomial
  of~$β$, $P$ divides $P_g$.
\end{proof}

We can now prove that if a Pisot numeration preserves zeros, then the
corresponding $β$-numeration satisfies the Condition~F.  
  Recall that we  defined $[g_n ⋯ g_0]_β = ∑_{i=0}^n g_iβ^i$ in
  Section~\ref{sec:notation}.

\begin{proposition}\label{prop:zero-preserving-then-F}
  Let $U = (u_n)_{n ⩾ 0}$ be a Pisot numeration associated to~$β$.  If the
  numeration~$U$ preserves zeros, then $β$ satisfies the Condition~F.
\end{proposition}
\begin{proof}
  Let $x = ∑_{i=0}^ka_iβ^i$ be an element of $ℤ[β]$ and let
  $c = \max \{ |a_i| : 0 ⩽ i ⩽ k\}$.  Let $K_c$ be the constant guaranteed
  by Definition~\ref{def:preserve-zeros}.  Let $M$ be an integer larger
  than each of those guaranteed by Lemmas \ref{lem:unormbnorm-take2}
  and~\ref{lem:dividesPw}.  We consider the finite sequences 
  $g = a_k⋯ a_00^{M+K_c}$ and $g' = \unor(g)$.  Since $\ord(g) ⩾ M+K_c$
  and $\|g\|_∞ ⩽ c$, $\ord(g') ⩾ M$. Since $[g]_U = [g']_U$ and
  $\ord(g-g') ⩾ M$, $[g]_β = [g']_β$ by Lemma~\ref{lem:dividesPw}.  Since
  $\ord(g') ⩾ M$ and $g'$ is $U$-normalised, it is also $β$-normalised.
  This shows that $β^{M+K_c}x$ belongs to $\fin(β)$ and thus $x ∈ \fin(β)$.
\end{proof}

\subsection{From Condition F to preservation of zeros}
\label{sec:condition-F-to-perservation}

The following lemma requires the additional assumption that $d_β(1)$ is
finite, an assumption which is fulfilled as soon as $β$ satisfies the
Condition~F. It tells us that, away from the radix point, the factors of
elements of~$ℕ_U$ and the $β$-shift agree, so that the recurrence defined
by $d_β(1)$ determines these languages.

The following lemma states that a finite sequence~$g ∈ \{0, 1\}^*$ where
the 1's are far enough apart is the normalised expansion of some integer,
that is, belongs to~$ℕ_U$.
\begin{lemma}\label{lem:spacedexpand}
  Let $U = (u_n)_{n ⩾ 0}$ be a Pisot numeration associated to~$β$.  There
  exists an integer~$K$ such that for each finite
  sequence~$g ∈ \{0, 1\}^*$, if $\ord(g) ⩾ K$ and if for all
  $0 ⩽ i < j < |g|$, $g_ig_j = 1$ implies $j-i ⩾ K$, then $g = (n)_U$ for
  some integer~$n$.
\end{lemma}
\begin{proof}
  By Lemma~\ref{lem:boundun}, there exist constants $K'$ and~$λ$ such
  that $|u_n - λ β^n| ⩽ K'\alphanorm^n$. Let $K$ be a constant to be
  chosen later and let $g$ be a finite sequence satisfying the hypotheses
  of the lemma.  In order to show that $g = (n)_U$ for some integer~$n$, it
  suffices to show that for each integer~$k > K$ the relation
  $∑_{i=0}^k g_iu_i < u_{k+1}$ holds.
  \begin{align*}
    ∑_{i=0}^k g_iu_i & ⩽ ∑_{i=0}^{⌊k/K⌋-1} u_{k-iK} \\
                     & ⩽  ∑_{i=0}^{⌊k/K⌋-1} \lambdaβ^{k-iK}  + K' \alphanorm^K/(1-\alphanorm^K) \\
                     & ⩽  \lambda(β^{k+2K}-β^{K})/(β^{K}-1) + K'\alphanorm^K/(1-\alphanorm^K) \\
                     & ⩽  \lambdaβ^k β^{K}/(β^K-1) + K'\alphanorm^K/(1-\alphanorm^K).
  \end{align*}
  Then we choose $K$ large enough such that $β^{K}/(β^K-1) < (3+β)/4$ and
  $K'\alphanorm^K/(1-\alphanorm^K) < \lambdaβ^K(β-1)/4$ and
  $u_{k+1}/β^k > \lambda(1+β)/2$ for each $k > K$.  The first inequality is
  possible because $β^K/(β^K-1)$ converges to~$1$ and $1 < (3+β)/4$.  The
  second inequality is possible because $K'\alphanorm^K/(1-\alphanorm^K)$
  converges to~$0$ while $\lambdaβ^K(β-1)/4$ converges to~$+∞$.  The last
  inequality is possible because $u_{k+1}/β^k$ converges to~$\lambdaβ$ and
  $(1+β)/2 < β$.  Finally we get that for each $k > K$,
  \begin{align*}
    ∑_{i=0}^k g_iu_i & ⩽ \lambdaβ^k(3+β)/4 + \lambdaβ^K(β-1)/4 \\
                    & ⩽ \lambdaβ^k (1+β)/2 < u_{k+1}.
                    \end{align*}
                    The result follows.
\end{proof}

The following lemma shows that each finite sequence can be decomposed as a
sum of normalised finite sequences where the number of components only
depends on the maximal value occurring in it.
\begin{lemma} \label{lem:decompsum}
  Let $U = (u_n)_{n ⩾ 0}$ be a Pisot numeration, and let $g$ be a finite
  sequence.  For each integer~$c ⩾ 0$, there exists an integer $k = k(c)$
  such that if $\|g\|_∞ ⩽ c$ and $\ord(g)⩾ K$, then there exist $2k$
  non-negative integers $n_1,…, n_k, m_1,…, m_k$ such that
  $g = (n_1)_U + ⋯ + (n_k)_U - (m_1)_U - ⋯ - (m_k)_U$.
\end{lemma}
\begin{proof}
  The result is trivial if $\|g\|_∞ = 0$. First suppose that all entries of
  $g$ are nonnegative. If $\|g\|_∞ = 1$, let $K$ be the integer provided by
  Lemma~\ref{lem:spacedexpand}.  For each $0 ⩽ j< K$, define the finite
  sequences $x^{(j)}$ by
  \begin{displaymath}
    x^{(j)}_i =
    \begin{cases}
      g_i & \text{if $i ≡ j \bmod K$ and $i ⩾ K$} \\
      0   & \text{otherwise.}
    \end{cases}
  \end{displaymath}
  We have $g = x^{(0)} + ⋯ + x^{(K-1)} $ since $\ord(g)⩾ K$.  It is also
  clear that each finite sequence~$x^{(j)}$ for $0 ⩽ j < K$ satisfies the
  hypothesis of Lemma~\ref{lem:spacedexpand}.  By taking
  $m_1= \cdots = m_k =0$, this completes the case $\|g\|_∞ = 1$ where all
  entries of $g$ are nonnegative. If $g$ also has negative entries,
    write $g=g^{+} - g^{-}$ where
  \begin{displaymath}
    g^{+} _i=
    \begin{cases}
      g_i & \text{ if $g_i ⩾ 0$} \\
      0 & \text{otherwise}
    \end{cases}
    \qquad\text{ and }\qquad
    g^{-} _i=
    \begin{cases}
      -g_i & \text{ if $g_i< 0$} \\
      0 & \text{otherwise.}
    \end{cases}
  \end{displaymath}
  Since $g^{+}$ and $g^{-}$ have only nonnegative entries, we can conclude
  that there exist $n_1,…, n_k, m_1,…, m_k$ such that
  $g^{+} = (n_1)_U + ⋯ + (n_k)_U $ and $g^{-} = (m_1)_U +⋯ +(m_k)_U$, and
  the result follows.  For $\|g\|_∞ > 1$, the proof is carried out by
  induction on~$\|g\|_∞$.
\end{proof}

In order to prove Proposition~\ref{pro:bnorm2unorm}, we need the following
result from~\cite[Prop.~2]{Frougny-Solomyak-1992}.  Let us denote by
$\fin_k(β)$ the set of numbers $x$ such that $\ord(\nor_β(x)) ⩾ -k$. More
formally, it is the set of numbers with normalised
$β$-expansion~$α = ∑_{i ⩽ n} a_iβ^i$ such that $a_i = 0$ for $i ⩽ -k$.
\begin{proposition}[Frougny-Solomyak \cite{Frougny-Solomyak-1992}]
  \label{pro:frougny-solomyak}
  Let $β$ be a Pisot number.  There exists an integer $ℓ = ℓ(β)$ having the
  following property.  If $x, y ∈ \fin_n(β)$ for some integer~$n$ and
  $x > y$, then $x+y ∈ \fin_{n+ℓ}(β)$ and $x-y ∈ \fin_{n+ℓ}(β)$.
\end{proposition}
It follows easily from this result that if the numbers $x_1,…,x_k$ belong
to~$\fin_n(β)$, then  any of the sums $x_1 \pm ⋯ \pm x_k$  belong to
$\fin_{n+⌈\log_2k⌉ℓ}$.

The following proposition states that the preservation of zeros is implied
by the Condition~F.  It follows that the preservation of zeros by
$U = (u_n)_{n⩾ 0}$ only depends on the recurrence and not on the initial
conditions defining~$U$.
\begin{proposition}\label{pro:bnorm2unorm}
  Let $U = (u_n)_{n ⩾ 0}$ be a Pisot numeration associated to~$β$.  If $β$
  satisfies the Condition~F, then~$U$ preserves zeros.
\end{proposition}
\begin{proof}
  Suppose that the Pisot number~$β$ satisfies the Condition~F, i.e.,
  satisfies $ℤ[β] ⊂ \fin(β)$, and let $c$ be a given integer.  Since $1$
  belongs to~$ℤ[β]$, it has a finite expansion $d_β(1) = t_1 ⋯ t_m$.  By
  Lemma~\ref{lem:decompsum}, there exists an integer~$k$ such that each
  finite sequence~$g$ with $\|g\|_∞ ⩽ c$ can be decomposed as a sum of
  finite sequences $g ∈ g_1 \pm ⋯ \pm g_k$ where each ~$g_i=(n_i)_U$ for
  some integer~$n_i$.  Furthermore it can be assumed that
  $\ord(g_i) ⩾ \ord(g)$.  Let $N$ be the integer provided by
  Lemma~\ref{lem:unormbnorm-take2}, and let $ℓ$ be the integer from
  Proposition~\ref{pro:frougny-solomyak}.  Let us suppose that
  $\ord(g) ⩾ N + ⌈\log_2 k⌉ℓ$.  Then there exists a $β$-normalised finite
  sequence~$g'$ such that $[g]_β = [g']_β$ and $\ord(g') ⩾ N$.  By
  Lemma~\ref{lem:unormbnorm-take2}, the finite sequence~$g'$ is also
  $U$-normalised and $[g]_U = [g']_U$.
\end{proof}

\begin{example}
  The classic Rauzy fractal is generated by the Tribonacci numeration, seen
  in Figure~\ref{fig:Tribonacci}. It is generated by taking the image under
  $\rval$ of the normalised expansions in $ℤ_U$, namely of
  several natural numbers for the standard initial conditions.
\begin{figure}[htbp]
  \centering
  \begin{subfigure}{0.3\textwidth}
    \centering
    \includegraphics[width=\linewidth]{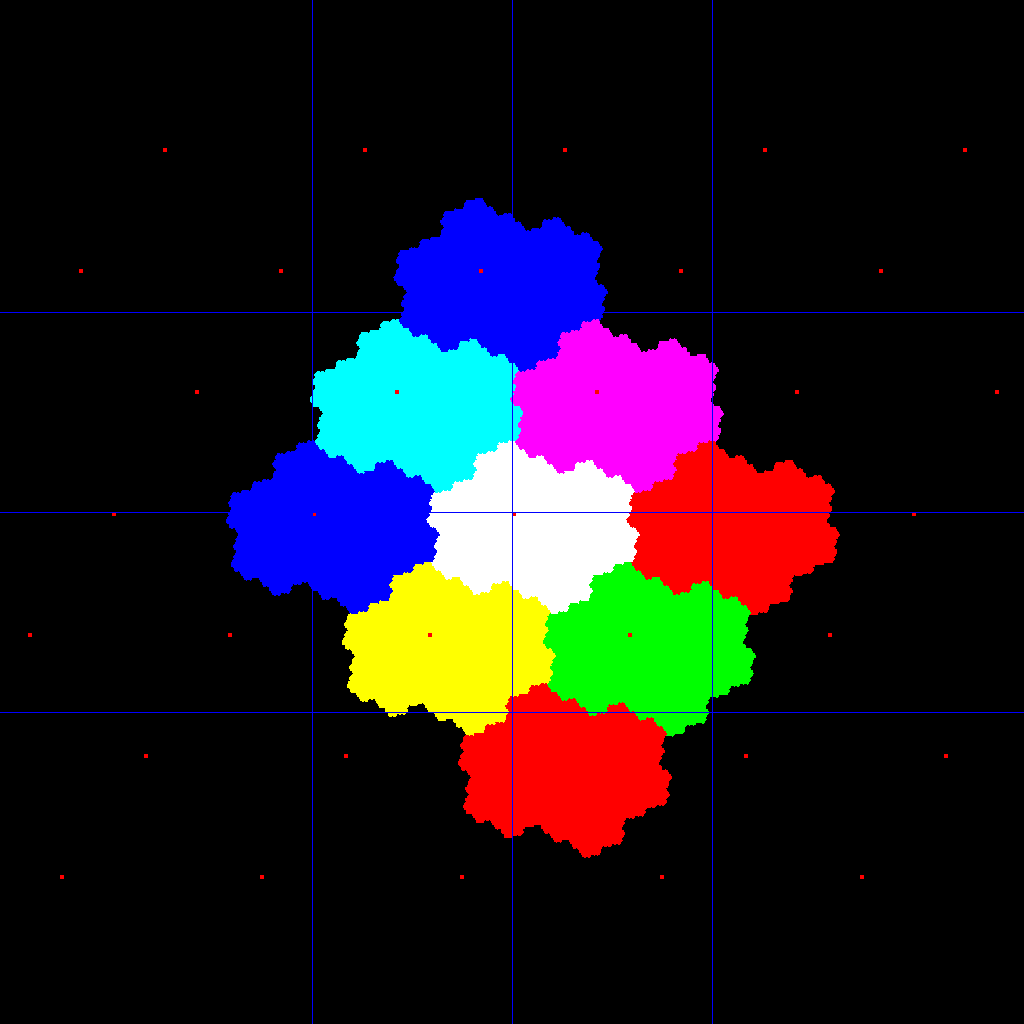}
  \end{subfigure}
  \hfill
  \caption{One tile is the image of  the normalised representatives
    of the group~$ℤ_U$ under~$\rval$ for the standard
    Tribonacci numeration. The red dots are the lattice~$\mathbb{L}$.}
  \label{fig:Tribonacci}
\end{figure}
\end{example}

\begin{example}\label{ex:initial-conditions}
  Consider the numeration generated by $P(X)=X^3-X^2 - 2X-1$.
  By~\cite[Theorem 3]{Akiyama-1998-cubic}, the Pisot root~$β$
  satisfies Condition~F, so by Proposition~\ref{pro:bnorm2unorm}, $U$
  preserves zeros. In Figure~\ref{fig:standard} we see how nonstandard
  initial conditions affect the tiling properties of~$ℂ$. If they
  are too small, we do not have a cover of~$ℂ$, and if they are
  too large, the tiles overlap.
  \begin{figure}[htbp]
  \centering
  \hfill
  \begin{subfigure}{0.3\textwidth}
    \centering
    \includegraphics[width=\linewidth]{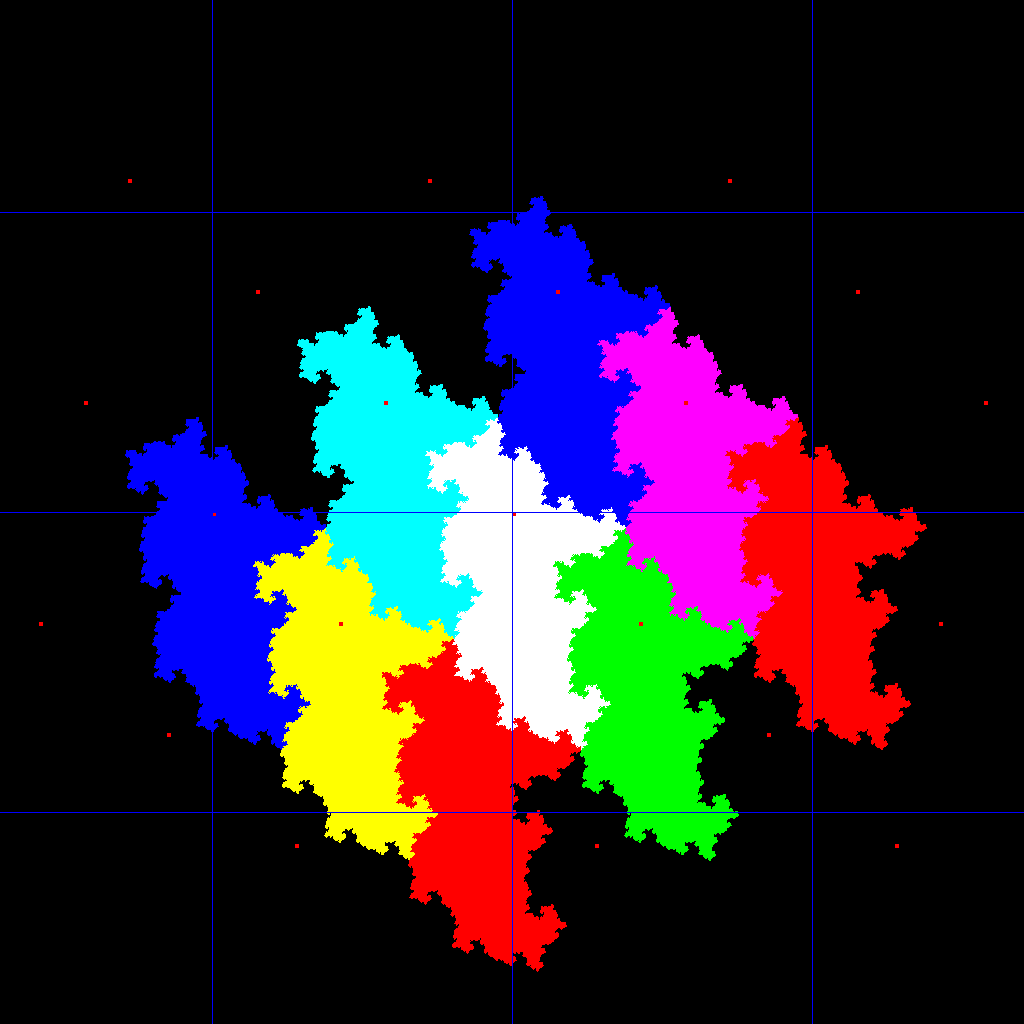}
  \end{subfigure}\hfill
  \begin{subfigure}{0.3\textwidth}
    \centering
    \includegraphics[width=\linewidth]{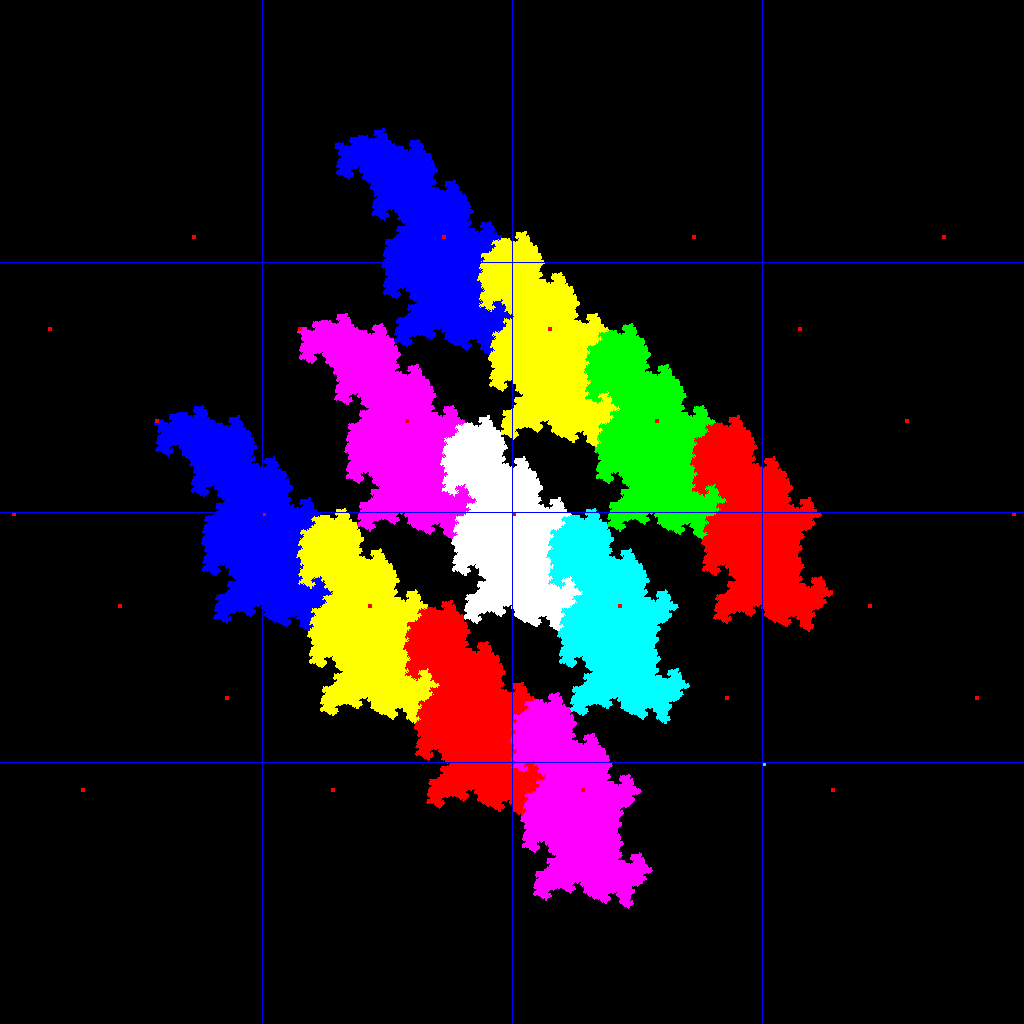}
  \end{subfigure}
  \hfill
  \begin{subfigure}{0.3\textwidth}
    \centering
    \includegraphics[width=\linewidth]{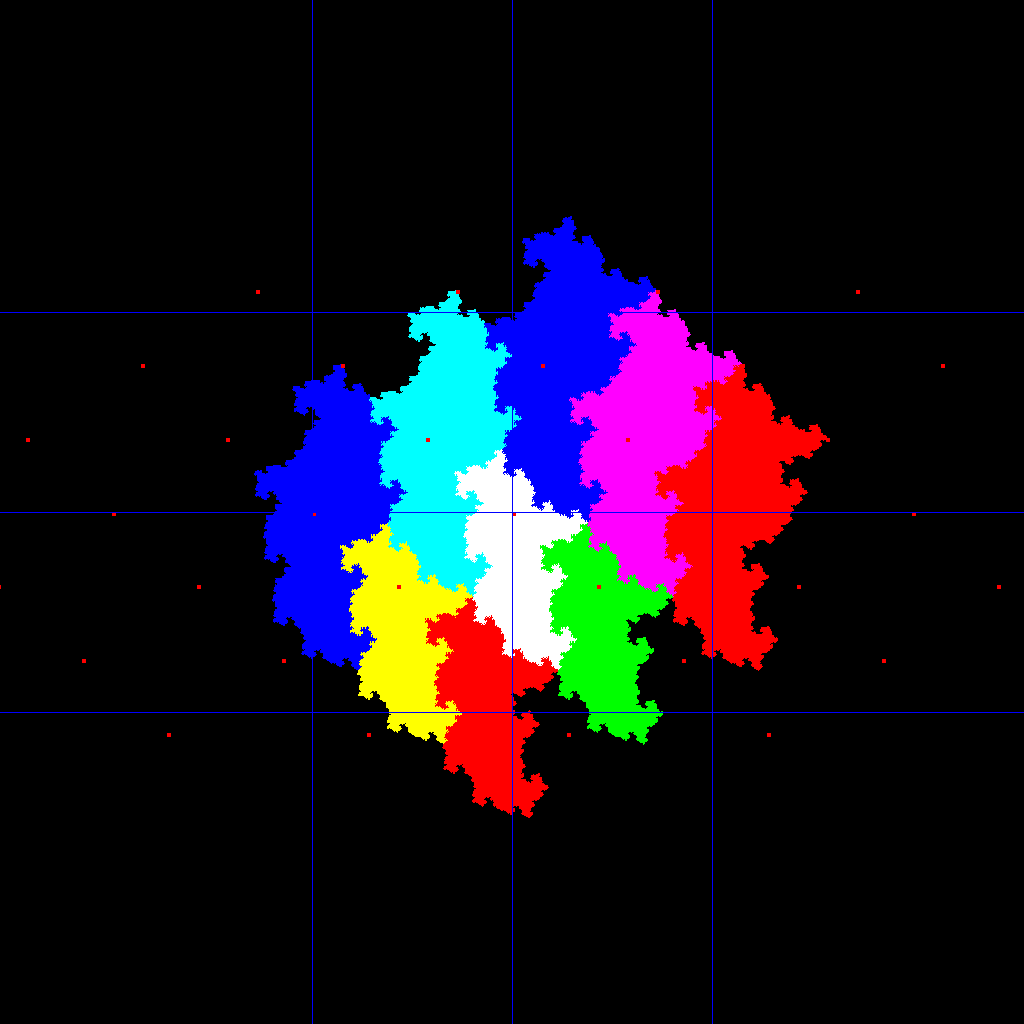}
  \end{subfigure}
  \caption{One tile in the left figure is the image of the normalised
    representatives of the group~$ℤ_U$ under~$\rval$ for the standard
    numeration given by $P(X)=X^3-X^2 - 2 X -1 $.  In the middle figure, a
    tile is the image of~$ℤ_U$ where $U$ is generated with non standard
    initial conditions $u_2 = 2$, $u_{1} = 1$, $u_{0} = 1$, on the right,
    with non-standard initial conditions $u_2 = 5$, $u_{1} = 2$, $u_{0} = 1$.}
  \label{fig:standard}
\end{figure}
\end{example}

\bibliographystyle{amsalpha}
\bibliography{fibonadic}

\end{document}